\documentclass[12pt]{amsart}
\usepackage{graphics,wrapfig, graphicx, mathrsfs,xcolor}
\usepackage{mathtools}
\usepackage{amssymb,amscd,graphicx,xcolor,subfig,tikz}
\usepackage{graphics}
\usepackage{indentfirst}
\usepackage{bm, enumerate,xcolor}
\usepackage[dvips]{epsfig}
\usepackage{latexsym}
\usepackage{float}
\usepackage{amsmath}
\usepackage{amssymb}
\newtheorem{lemma}{Lemma}[section]
\newtheorem{theorem}{Theorem}[section]
\newtheorem{corollary}{Corollary}[section]

\newtheorem{remark}{Remark}[section]

\numberwithin{equation}{section} \numberwithin{theorem}{section}
\numberwithin{example}{section} \numberwithin{remark}{section}
\numberwithin{figure}{section} \numberwithin{algorithm}{section}
\mathtoolsset{showonlyrefs}

\def\ba{\begin{array}}
\def\ea{\end{array}}
\def\bma{\left(\begin{matrix}}
\def\ema{\end{matrix}\right)}
\def\be{\begin{equation}}
\def\ee{\end{equation}}
\newcommand{\yupei}[1]{{#1}}

\def\dfrac{\displaystyle\frac}

\begin{document}

\title{Regular and Singular Steady States of the 2D  incompressible Euler equations near the Bahouri-Chemin Patch}
\author{Tarek M. Elgindi, Yupei Huang}

\begin{abstract}
We consider steady states of the two-dimensional incompressible Euler equations on $\mathbb{T}^2$ and construct smooth and singular steady states around a particular singular steady state. More precisely, we construct families of smooth and singular steady solutions that converge to the Bahouri-Chemin patch.
\end{abstract}

\maketitle
\section{Introductions and Notations}
\subsection{Introduction}
We study spatially periodic solutions to the 2d incompressible Euler equations. The equations are expressed in vorticity form as: 
\begin{equation}\label{1}
\begin{aligned}
    &\partial_t\omega+u\cdot \nabla \omega=0,\\
    &u=\nabla^{\perp} \triangle^{-1}\omega,
    \end{aligned}
\end{equation}
where $\nabla^{\perp}=(-\partial_y,\partial_x)$. \par
It is well-known by the Beale-Kato-Madja criterion \cite{Beale1984RemarksOT} that when the vorticity is initially smooth, the solution to \eqref{1} stays smooth for all time. Moreover, there is a double exponential upper bound for the vorticity gradient \cite{wolibner1933theoreme,YUDOVICH19631407}. An important open question is whether the double exponential upper bound of \yupei{vorticity gradient} is sharp. There are no known improvements \yupei{on} the upper bound in general, and most recent works in this direction concern specific geometric scenarios.\par
In the direction of improving the upper bounds, the authors of \cite{Hoang2014NoLD} demonstrated that the vorticity gradient can only increase exponentially for a limited (and possibly empty) set of solutions. In \cite{Bedrossian2013InviscidDA}, the authors demonstrate the presence of an open set of solutions with high regularity on $\mathbb{T}\times\mathbb{R}$ that increase at most linearly and provide a comprehensive account of the long-term behavior. For further information on the topic, one can refer to \cite{do2019vorticity,elgindi2020symmetries,itoh2016growth}, which provides evidence that double exponential merging of particles at the origin is not possible under certain symmetry conditions and in different domains.\par  In the direction of seeking solutions to the incompressible Euler equations that admit \yupei{vorticty gradient growth}, the authors generally study solutions near steady states.  The author of \cite{Denisov2008InfiniteSG} proved superlinear growth of the vorticity gradient for certain solutions near the steady state $\omega^*(x,y)=\sin{x}\sin{y}$ in $\mathbb{T}^2$. A key feature of this steady state is that the associated velocity field induces exponential expansion and contraction near the fixed point $(0,0).$ In the follow-up work \cite{Denisov2012DoubleEG}, the author constructed solutions on the $\mathbb{T}^2$ that admit double exponential growth for the vorticity gradient for a fixed time interval by perturbing a different steady state: the Bahouri-Chemin patch, which satisfies the following equation:\begin{align}
    \Delta \psi_0(x,y)=-\text{sgn}(x)\text{sgn}(y),\yupei{\quad \forall (x,y) \in \mathbb{T}^2:=\left[-\frac{1}{2},\frac{1}{2}\right]^2\bigg/ \sim.} 
\end{align} In the important work \cite{Kiselev2013SmallSC}, the authors proved that smooth solutions for the incompressible Euler equations in the disk can indeed experience double exponential growth of the gradient of vorticity, again by perturbing a version of the Bahouri-Chemin patch. In that work, the presence of a physical boundary is essential to get information on the long time dynamics of the flow field near the origin and it remained open whether double exponential growth could happen on $\mathbb{T}^2$. The author of \cite{zlatovs2015exponential} extended some of the ideas of \cite{Kiselev2013SmallSC} to give exponential growth for $C^{1,\alpha}$ solutions on $\mathbb{T}^2$ and exponential growth of the second derivative of vorticity for smooth solutions. Later, the methods introduced in \cite{Kiselev2013SmallSC} were further developed to establish double exponential growth in other settings in the presence of a physical boundary (see \cite{Kiselev2019} and \cite{xiaoqian2014}).  See also \cite{choi2022infinite,drivas2022singularity} for more recent results in the direction of long-time growth. \par  The purpose of this work is to investigate smooth and singular steady states near the Bahouri-Chemin patch on $\mathbb{T}^2$. The existence of such smooth solutions shows that  constructing smooth solutions that grow double exponentially may be significantly more subtle on $\mathbb{T}^2$ than on domains with a boundary. Our starting point is to consider the equation for steady states of the 2d incompressible Euler equations:
\begin{align}\label{steady}
    \nabla^{\bot}\psi\cdot \nabla \triangle \psi=0,
\end{align}
where $\psi$ is the stream function, $u=\nabla^{\bot}\psi$ is the velocity, and  $\omega=\triangle \psi$ is the vorticity.
It is clear that if there exists $F\in C^{1}(\mathbb{R})$ such that $\triangle\psi=F(\psi)$, then \eqref{steady} holds automatically. The above fact provides a way to find steady states of the 2d incompressible Euler equations. More specifically, if we can find a smooth function $F$ and a $\psi$ satisfying  $\psi=\triangle^{-1}F(\psi)$,  $\psi$ would correspond to the stream function for a steady state of incompressible Euler equations. \par In \cite{Choffrut2012LocalSO}, a 1-to-1 correspondence between $\psi$ and $F$ is established near arbitrary steady states $(\psi^*,F_0)$ satisfying certain non-degeneracy conditions. See also \cite{constantin2021flexibility,wirosoetisno2005persistence,Zelati2020StationarySN} for other related results in this direction. 
A key consideration in all of the above works is the mapping properties of the linear operator:
\[
\mathcal{L}_*:=\Delta-F_{*}'(\psi_*),
\]
since most of these works rely on arguments related to the implicit function theorem.
In particular, in the works mentioned above, it is assumed that $\psi_*$ and the perturbations thereof are regular enough so that it is possible to use the linearized operator $\mathcal{L}_*$. In cases where $F_*$ is not differentiable, it is not clear how to replace the use of $\mathcal{L}_*$. In particular, in the case of Bahouri-Chemin patch, which satisfies the following semi-linear elliptic equation:\begin{align}
    \label{equ}
    \Delta \psi_0(x,y)=-\text{sgn}(\psi_0)(x,y),\yupei{\quad \forall (x,y) \in \mathbb{T}^2},
\end{align}the jump discontinuity in -sgn makes it unclear how to proceed. Despite this difficulty, we are able to construct both smooth and singular steady states near $\psi_0$.  

\subsection{Main results}
We now state the main results \yupei{of} our paper.
\yupei{The first theorem concerns smooth steady states near the Bahouri-Chemin patch.}
\begin{theorem}\label{main}
 There \yupei{exists a fixed constant} $\epsilon_0>0$, such that \yupei{for any $\epsilon \in (0,\epsilon_0)$} , we can find a family of odd smooth \yupei{functions} $F_{\epsilon}$, \yupei{together with} a family of smooth and odd-odd\footnote{This means that $\psi_\epsilon(-x,y)=-\psi_\epsilon(x,y)=\psi_\epsilon(x,-y),$ for all $\epsilon,x,y.$} \yupei{functions} $\psi_{\epsilon}$ \yupei{such that the following relation holds}:\yupei{ \begin{equation}
  \psi_{\epsilon}(x,y)=\triangle^{-1}F_{\epsilon}(\psi_{\epsilon}(x,y)), \quad (x,y) \in \mathbb{T}^2=\left[-\frac{1}{2},\frac{1}{2}\right]^{2}\bigg/ \sim.
\end{equation}}
In addition, \yupei {we have the following properties:
\begin{itemize}
\item {The $L^\infty$-norm of   $\triangle{\psi_{\epsilon}}$ is bounded by $1$, i.e., $\|\triangle{\psi_{\epsilon}}\|_{L^{\infty}(\mathbb{T}^2)}\leq 1.$}
\item{ For any $\epsilon \in (0,\epsilon_0)$, the function $\psi_{\epsilon}(x,y)$ is jointly smooth in all its variables:
$\displaystyle \psi_{(\cdot)}(\cdot,\cdot)\in C^{\infty}((0,\epsilon_0)\times\mathbb{T}^2)$}.
\item When the parameter $\epsilon$ approaches $0$, the sequence of the steam functions $\psi_{\epsilon}$ converges to the Bahouri-Chemin patch $\psi_0$ in $C^{1,\alpha}(\mathbb{T}^2):$ 
\begin{equation*}
    \lim_{\epsilon\rightarrow 0}\|\psi_{\epsilon}-\psi_0\|_{C^{1,\alpha}(\mathbb{T}^2)}=0,
\end{equation*}
for any $\alpha\in [0,1)$.
\end{itemize}}
\end{theorem}

\begin{remark} 
Since $\Delta\psi_0$ itself is discontinuous, the convergence of $\psi_{\epsilon} $ to $\psi_0$ cannot happen in $C^2$ or uniformly in the vorticity.
The argument to construct smooth steady states near $\psi_0$ can be easily extended to construct smooth steady states near corresponding Bahouri-Chemin patches on the disk whose vorticity {\bf vanishes on the boundary}. Note that a consequence of \cite{Kiselev2013SmallSC} is that any odd smooth solution on the disk close to the Bahouri-Chemin patch, with non-constant vorticity on the boundary, must experience infinite gradient growth.   
\end{remark}
\yupei{Our second theorem deals with singular steady states near the Bahouri-Chemin patch.}
\begin{theorem}\label{main2}
There  \yupei{exists a fixed constant $\epsilon_0>0$, such that for any fixed $\epsilon \in (0,\epsilon_0)$, there is a steady state $\phi_{\epsilon}$ whose vorticity has algebraic singularities on the separatrices $\displaystyle\left\{x={k}/{2}\right\}_{k\in\mathbb{Z}}\cup \left\{y={k}/{2}\right\}_{k\in\mathbb{Z}}\subseteq \mathbb{T}^2=\left[-{1}/{2},{1}/{2}\right]^2\big/\sim$. Moreover, these singular steady states approach the Bahouri-Chemin patch $\psi_0$ as $\epsilon$ goes to $0$, i.e.,
\begin{equation}
\lim_{\epsilon\rightarrow 0} \|\phi_{\epsilon}-\psi_0\|_{C^1(\mathbb{T}^2)}=0.
\end{equation} Furthermore, particles transported by the velocity field of these singular steady states may hit the origin in finite time.}  
\end{theorem}
\yupei{\begin{remark}
    We recall that the particles transported by the flow generated by the Bahouri-Chemin patch can only approach the origin at most double-exponentially. This is in strong contrast to our result in Theorem \ref{main2}.
\end{remark}}

\subsection{Main ideas of the paper}

To find smooth steady states near the Bahouri-Chemin patch, we study the fixed points of the operator $T_\epsilon:=\Delta^{-1}(F_{\epsilon})$ in a certain subset of $C^{1}(\mathbb{T}^2)$. While previous results in this direction usually invoked versions of the Banach fixed-point theorem, we rely here on the Schauder fixed-point theorem, which seems to be more well-adapted to singular situations. In particular, we do not need to directly analyze any linearized operator. Toward applying the Schauder fixed point theorem, the goal will be to find a closed, convex, and bounded set that is invariant under the mapping $T_\epsilon.$. 

The existence of an invariant set for $T_\epsilon$ is largely based on the analysis of the geometry of the near-zero level sets of $\psi_0$. It is known from \cite{Bahouri1994EquationsDT} that \[C |xy| \ln\frac{1}{x^2+y^2}<|\psi_0(x,y)|,\] in a neighborhood of $(0,0),$ for some fixed constant $C>0.$ Due to the super quadratic growth of $|\psi_0|$ near the origin, we prove that the region where $|\psi_0|<\epsilon$ has an area approximately $\epsilon \ln\frac{1}{\epsilon}$ as $\epsilon\rightarrow 0.$ This observation is key in determining the right bounds for the sought invariant set of $T_\epsilon.$ A proper choice of $F_\epsilon$ combined with potential theory estimates guarantees that a suitably chosen small neighborhood of $\psi_0$ in $C^1$ is invariant under the map $\Delta^{-1}F_{\epsilon}$. \par
In the direction of finding singular steady states near the Bahouri-Chemin patch, we seek solutions for $\Delta\psi_\epsilon=F_\epsilon(\psi_\epsilon),$
where the forcing term $F_{\epsilon}$ has an algebraic singularity.  
More specifically, when $F_{\epsilon}(x)$ is an odd function such that \begin{align*}
    F_{\epsilon}(x)=&\left\{\begin{array}{cc}
         -\frac{\epsilon^s}{x^s}, &\quad 0<x<\epsilon,\\
         -1, &\quad x\geq \epsilon.
    \end{array}\right.
\end{align*}
It is not difficult to extend the proof of Theorem \ref{main} to show that corresponding solutions $\psi_\epsilon$ exist and converge to $\psi_0$ when $\epsilon\rightarrow 0$ in the case where $0<s<\frac{1}{2}.$ We are also able to do this in the case $\frac{1}{2}\leq s<1,$ but the proof requires more work. In contrast to the proof of Theorem \ref{main} and case $s<\frac{1}{2}$, we cannot simply rely on lower bounds of $|\psi_0|$ when $s\geq \frac{1}{2}.$ Here, we have to use a barrier argument (inspired by \cite{Pino1992AGE}) to establish $s$-dependent \emph{a-priori} lower bounds on $\psi_\epsilon$ that can be used to close the fixed point scheme. 

\subsection{Organization of the paper}
The rest of the paper will be organized as follows. In Section 2, we will prove Theorem \ref{main}. We will start by giving some technical lemmas and then prove the results concerning existence in Theorem \ref{main}. Then we will discuss smooth dependence on $\epsilon$ when $\epsilon>0$ and finish the proof of the Theorem \ref{main}. In Section 3, we will prove Theorem \ref{main2}. In Appendix A, we collect a few technical facts that we need in the proofs. 
\subsection{Notation}
Throughout this paper, we will reserve some characters for certain quantities according to the following rules:
\begin{itemize}
\item $\mathbb{T}^2$: $[-\frac{1}{2},\frac{1}{2}]^2/{\sim}$, where $(x_1,y_1)\sim (x_2,y_2)$ if there is a pair $(m,n)\in \mathbb{Z}^2,(x_1,y_1)=(x_2+m,y_2+n)$.
\item $\mathcal{T}$: the triangle region $\left\{0<y<x<\frac{1}{4}\right\}.$
\item $\mathbb{T}^{2+}$: the first quadrant of flat torus $\left\{0<x<\frac{1}{2},0<y<\frac{1}{2}\right\}$
\item $\mathbb{R}^{2+}$: the first quadrant in $\mathbb{R}^2$ $\left\{x>0,y>0\right \}$
    \item $B_\delta (x)$: the ball centered at $x$ with radius $\delta$. When $x=0,$ we simply write $B_\delta$.
 \item $K_{\delta}$: an approximation of the identity. $K_{\delta}(x)=\frac{1}{\delta}K(\frac{x}{\delta})$, where $K(x)$ is a smooth non-negative even function with support in $(-1,1)$, and $\int_{\mathbb{R}} K(x)dx=1$.
 \item $\epsilon:$ the continuation parameter.
  \item $\epsilon_0$: the positive constant given in Theorem \ref{main}. 
    \item $C$: Generic postive constant that can change from line to line, which is independent of $\epsilon_0$.
    \item $C_1$: A positive constant used to define the invariant set for $T_\epsilon.$
     \item $C_2$: Generic positive constant depending on $\epsilon$. 
        \item $\chi_{A}(x)$: the characteristic function of a set A.
    \item $\text{sgn}(x)$:
    \begin{align*}
    \text{sgn}(x)=&\left\{\begin{array}{cc}
         -1, &\quad x<0,\\
         1, &\quad x\geq  0.
    \end{array}\right.
\end{align*}
    \item $G$: the Green's function of the Laplacian on the torus. See Lemma \ref{GreensFunction}.
    \item $\Delta^{-1}$: $\Delta^{-1}f=f*G.$
       \item $\psi_{0}:$ \[ \psi_0=\triangle^{-1}[-\text{sgn}(x)\text{sgn}(y)],\] the stream function of the Bahouri-Chemin patch.
\end{itemize}
\section{The smooth steady states near $\psi_0$}
\yupei{The purpose of this section is to give a proof of Theorem \ref{main}.
We will firstly give the definition of $F_{\epsilon}$, and then divide the proof of Theorem \ref{main} into two parts: \begin{itemize}
    \item The existence of smooth steady states near $\psi_0$ associated to $F_\epsilon$,
    \item The smooth dependence of $\psi_{\epsilon}$ on $\epsilon$.
\end{itemize}
}
\subsection*{The definition and basic properties of $F_{\epsilon}$}\label{Construction1}\
We choose $F_{\epsilon}(x)=-sgn(.)*K_{\epsilon}(x)$. It is not difficult to show the following Lemma, whose proof relies only on the definition of convolution. 
\begin{lemma}\label{FProperties}
We have that $F_\epsilon\in C^\infty$ is odd and non-increasing. Moreover, we have that 
\[|F_\epsilon|_{L^\infty}\leq 1\qquad |F_\epsilon'|_{L^\infty}\leq \frac{C}{\epsilon}.\] If $x<-2\epsilon,$ $F_\epsilon(x)=1.$ For $\epsilon>0,$ $F_\epsilon$ is smooth in $C^k(\mathbb{R})$ for every $k.$ If $\psi_\epsilon$ is odd-odd, then $\Delta^{-1}(F_\epsilon(\psi_\epsilon))$ is also odd-odd. 
\end{lemma}

\subsection{On the existence of smooth steady states near $\psi_0$}
In this section, our main purpose is to prove the theorem below, which shows we can approximate $\psi_0$ by a sequence of smooth steady states in $C^{1,\alpha}(\mathbb{T}^2)$, for $\alpha \in[0,1)$:
\begin{theorem}\label{step1}
There is \yupei{a constant} $\epsilon_0>0$, such that \yupei{for any $\epsilon \in(0,\epsilon_0)$,} we can find a smooth function $\psi_{\epsilon}$ satisfying \begin{equation}\label{semi-linearequation} \psi_{\epsilon}=\triangle^{-1}(F_{\epsilon}(\psi_{\epsilon})).
\end{equation}
Moreover, we have the following properties:\yupei{ \begin{itemize}
    \item 
    $\|\psi_{\epsilon}-\psi_0\|_{C^{1}(\mathbb{T}^2)}\leq C\sqrt{\epsilon \ln{\frac{1}{\epsilon}}},$
    \item $|\Delta \psi_{\epsilon}|<1$,
    \item For $\alpha\in [0,1),$  we have
\begin{equation}
    \lim_{\epsilon\rightarrow 0}\ ||\psi_{\epsilon}-\psi_0\|_{C^{1,\alpha}(\mathbb{T}^2)}=0.
\end{equation}
\end{itemize}}
\end{theorem} 
As discussed in the introduction, our goal is to prove that the following non-linear map $T_{\epsilon}\psi=\triangle^{-1}(F_{\epsilon}(\psi))$ has a fixed point.  
 In the proof of Theorem \ref{main}, we will use the Schauder fixed point theorem: 
\begin{theorem}[Schauder 1930]\label{Shauder}
Let M be a closed, convex and bounded set in a Banach space X, if $T:M \rightarrow M$ is compact, then T has a fixed point. 
\end{theorem}
We notice the fact that $\|F_\epsilon\|_{L^\infty}\leq 1$ as in Lemma \ref{FProperties}, then by the properties of the Green function in Lemma \ref{GreensFunction}, $T_{\epsilon}$ is certainly a compact map from the space of $C^1$ odd-odd functions to itself. The key now to prove Theorem \ref{step1} is to find a \emph{bounded} invariant subset. This is the content of the following Lemma.
\begin{lemma}\label{lemma0}
There is a $\epsilon_0>0$ and  a constant $C_1$, such that \yupei{for any $\epsilon \in (0,\epsilon_0)$}, we have that the set 
\begin{align*}
  \mathcal{K}_\epsilon:=  &\left\{\psi\in C^1\big | \, \|\psi-\psi_0\|_{C^{1}(\mathbb{T}^2)}\leq C_1\sqrt{\epsilon\ln\frac{1}{\epsilon}}\right\}\\& \quad \cap \left\{\psi \big | \, \psi(x,y)=\psi(-x,-y)=-\psi(-x,y)=-\psi(x,-y)=\psi(y,x) \right\}
\end{align*} is invariant under the mapping $T_{\epsilon}$.  
\end{lemma} To prove that $T_\epsilon:\mathcal{K}_\epsilon\rightarrow\mathcal{K}_\epsilon,$ we need to estimate 
$\|T_\epsilon\psi-\psi_0\|_{C^1}.$ To do this, we write: 
\[\| T_\epsilon\psi- \psi_0\|_{C^1}=\| \Delta^{-1}\big(F_\epsilon(\psi)+\text{sgn}(\psi)\big)\|_{C^1}=\| \Delta^{-1}\big(F_\epsilon(\psi)+\text{sgn}(\psi)\big)\chi_{ |\psi |<2\epsilon})\|_{C^1}.\] In the last equality, we use that $F_\epsilon(x)=-\text{sgn}(x)$, when $|x|>2\epsilon$. 
It then becomes clear that we need to study the size of the set $\left\{ |\psi |<2\epsilon\right\}$ for any element $\psi\in \mathcal{K}_\epsilon.$ To do this, we start by studying the size of $\left\{ |\psi_0 |<2\epsilon\right\}.$ Application of a Steiner-type estimate and a judicious choice of $C_1$ then allows us to conclude the invariance of $\mathcal{K}_\epsilon.$ 
\subsection{Properties of the Bahouri-Chemin patch and some technical lemmas}
We have the following Lemma. 
\begin{lemma}\label{PsiProperties}
$\psi_0\in C^1(\mathbb{T}^2)$ is odd-odd and positive on $\mathbb{T}^{2+}:=(0,\frac{1}{2})^2.$ $\psi_0$ is also symmetric with respect to the diagonal line $y=x$ and the lines $x=\frac{1}{4}$ and $y=\frac{1}{4}$. 
Moreover, for $(x,y)\in \mathcal{T}:=\left\{0<y<x<\frac{1}{4}\right\},$ we have that 
    \begin{equation}\label{asytop}
       \frac{1}{C} xy\ln\frac{1}{x}<\psi_0(x,y)<C xy\ln\frac{1}{x}. 
    \end{equation}
    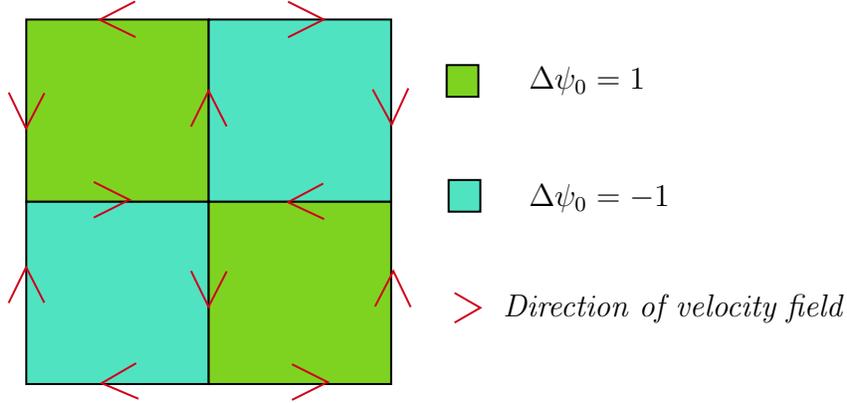
\begin{figure}[h]
        \centering

\tikzset{every picture/.style={line width=0.75pt}} 

\begin{tikzpicture}[x=0.75pt,y=0.75pt,yscale=-1,xscale=1]

\draw  [fill={rgb, 255:red, 126; green, 211; blue, 33 }  ,fill opacity=1 ] (199,55) -- (291,55) -- (291,147) -- (199,147) -- cycle ;
\draw  [fill={rgb, 255:red, 80; green, 227; blue, 194 }  ,fill opacity=1 ] (291,55) -- (383,55) -- (383,147) -- (291,147) -- cycle ;
\draw  [fill={rgb, 255:red, 80; green, 227; blue, 194 }  ,fill opacity=1 ] (199,147) -- (291,147) -- (291,239) -- (199,239) -- cycle ;
\draw  [fill={rgb, 255:red, 126; green, 211; blue, 33 }  ,fill opacity=1 ] (291,147) -- (383,147) -- (383,239) -- (291,239) -- cycle ;
\draw  [color={rgb, 255:red, 208; green, 2; blue, 27 }  ,draw opacity=1 ] (233,137) -- (251,146) -- (233,155) ;
\draw  [color={rgb, 255:red, 208; green, 2; blue, 27 }  ,draw opacity=1 ] (331,46) -- (349,55) -- (331,64) ;
\draw  [color={rgb, 255:red, 208; green, 2; blue, 27 }  ,draw opacity=1 ] (333,229) -- (351,238) -- (333,247) ;
\draw  [color={rgb, 255:red, 208; green, 2; blue, 27 }  ,draw opacity=1 ] (391.69,89.7) -- (383.31,107.99) -- (373.7,90.31) ;
\draw  [color={rgb, 255:red, 208; green, 2; blue, 27 }  ,draw opacity=1 ] (374.86,199.86) -- (384.14,182) -- (392.86,200.14) ;
\draw  [color={rgb, 255:red, 208; green, 2; blue, 27 }  ,draw opacity=1 ] (349.22,155.77) -- (331,147.23) -- (348.77,137.78) ;
\draw  [color={rgb, 255:red, 208; green, 2; blue, 27 }  ,draw opacity=1 ] (282.04,109.04) -- (290.96,91) -- (300.04,108.96) ;
\draw  [color={rgb, 255:red, 208; green, 2; blue, 27 }  ,draw opacity=1 ] (300.09,182.09) -- (290.91,200) -- (282.09,181.91) ;
\draw  [color={rgb, 255:red, 208; green, 2; blue, 27 }  ,draw opacity=1 ] (255.55,246.42) -- (237.02,238.56) -- (254.42,228.45) ;
\draw  [color={rgb, 255:red, 208; green, 2; blue, 27 }  ,draw opacity=1 ] (190.08,198.08) -- (198.92,180) -- (208.08,197.92) ;
\draw  [color={rgb, 255:red, 208; green, 2; blue, 27 }  ,draw opacity=1 ] (254.12,63.88) -- (236,55.12) -- (253.88,45.88) ;
\draw  [color={rgb, 255:red, 208; green, 2; blue, 27 }  ,draw opacity=1 ] (208.15,92.15) -- (198.85,110) -- (190.15,91.85) ;
\draw  [fill={rgb, 255:red, 126; green, 211; blue, 33 }  ,fill opacity=1 ] (411,78) -- (427,78) -- (427,94) -- (411,94) -- cycle ;
\draw  [fill={rgb, 255:red, 80; green, 227; blue, 194 }  ,fill opacity=1 ] (412,136) -- (428,136) -- (428,152) -- (412,152) -- cycle ;
\draw  [color={rgb, 255:red, 208; green, 2; blue, 27 }  ,draw opacity=1 ] (415,193) -- (428,200.5) -- (415,208) ;

\draw (451,77.4) node [anchor=north west][inner sep=0.75pt]    {$\Delta \psi _{0} =1$};
\draw (451,136.4) node [anchor=north west][inner sep=0.75pt]    {$\Delta \psi _{0} =-1$};
\draw (438,192) node [anchor=north west][inner sep=0.75pt]   [align=left] {Direction of velocity field};

\end{tikzpicture}

\caption{Illustrative figure for $\psi_0$}
        \label{fig:enter-label}
    \end{figure}
\end{lemma}
We will give the proof of these properties in the Appendix. 
As we have mentioned, a key step in our argument is to estimate the area of the set $\left\{ |\psi |<2\epsilon\right\}$ for $\psi\in\mathcal{K}_\epsilon.$ We start with the following key Lemma about $\psi_0.$ \begin{lemma}[Level set estimate for $\psi_0$]\label{level set}
If $A$ is sufficiently small, we have that
\begin{equation}\label{levelset123}
    \big |\left\{|\psi_0|<A\right\}\big |\leq C A \ln \ln \frac{1}{A}.
\end{equation}
\end{lemma}

\begin{proof}
By the symmetries in Lemma \ref{PsiProperties}, we may consider the points of $\left\{(x,y):|\psi_0(x,y)|<A\right\}$ in $\mathcal{T}$. In this region, again by Lemma \ref{PsiProperties}, we have that \[0<y<L_{A}(x):=\text{min}\left\{\frac{C\cdot A}{x\ln\frac{1}{x}}, x\right\}.\] 
It follows that the area of the region in question is bounded by
\[C\int_{0}^{\frac{1}{4}} L_{A}(x)\leq CA\ln\ln \frac{1}{A}.\]
 The result follows.
 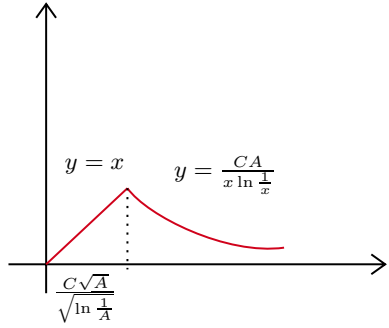
\begin{figure}[h]
    \centering

\tikzset{every picture/.style={line width=0.75pt}} 

\begin{tikzpicture}[x=0.75pt,y=0.75pt,yscale=-1,xscale=1]

\draw  (160,179.4) -- (350,179.4)(179,48) -- (179,194) (343,174.4) -- (350,179.4) -- (343,184.4) (174,55) -- (179,48) -- (184,55)  ;
\draw [color={rgb, 255:red, 208; green, 2; blue, 27 }  ,draw opacity=1 ]   (220,141) -- (179,179.4) ;
\draw [color={rgb, 255:red, 208; green, 2; blue, 27 }  ,draw opacity=1 ]   (220,141) .. controls (230,155) and (270,175) .. (299,171) ;
\draw  [dash pattern={on 0.84pt off 2.51pt}]  (220,141) -- (220,182) ;

\draw (187,124.4) node [anchor=north west][inner sep=0.75pt]  [font=\scriptsize]  {$y=x$};
\draw (242,123.4) node [anchor=north west][inner sep=0.75pt]  [font=\scriptsize]  {$y=\frac{CA}{x\ln\frac{1}{x}}$};
\draw (181.04,182.77) node [anchor=north west][inner sep=0.75pt]  [font=\tiny,rotate=-359.26]  {$\frac{C\sqrt{A}}{\sqrt{\ln\frac{1}{A}}}$};

\end{tikzpicture}

    \caption{Graph of $L_{A}(x)$}
\end{figure}
\end{proof}
Next, we will do a similar level-set estimate for $\psi\in\mathcal{K}_\epsilon.$
\begin{lemma}[Key Estimate for $\psi\in\mathcal{K}_\epsilon$]\label{level set estimate2}
    For any $\psi\in \mathcal{K}_{\epsilon}$, we have \begin{equation}
        \big |\left\{(x,y): |\psi(x,y)|<2\epsilon\right\}\big |\leq C\Big( \epsilon \ln\ln\frac{1}{\epsilon}+C_1\frac{\epsilon}{\ln{\frac{1}{\epsilon}}}\Big),
    \end{equation} where $C_1$ is as in the definition of $\mathcal{K}_\epsilon.$
\end{lemma}
Restricting to $\mathcal{T},$ to prove Lemma \ref{level set estimate2}, we split the region $\left\{|\psi|<2\epsilon\right\}$ into the region $\left\{\psi<\frac{1}{2}\psi_0\right\}$ and the region $\left\{\psi>\frac{1}{2}\psi_0\right\}$. The first region is considered in the following Lemma.  
\begin{lemma}[Small error estimate]\label{small error}
Let  $\psi$ be a fixed $C^1$ function with  
    $\|\psi-\psi_0\|_{C^{1}(\mathbb{T}^2)}=B$ sufficiently small,
 then
\begin{equation*}
    \Omega_{B}:=\left\{(x,y)\in\mathcal{T}:(\psi_0-\psi)(x,y)>\frac{1}{2}\psi_0(x,y)\right\}\subseteq \left\{0<y<x<\frac{CB}{\ln{\frac{1}{B}}}\right\}.
\end{equation*}
\end{lemma}
\begin{figure}[h]
    \centering

\tikzset{every picture/.style={line width=0.75pt}} 

\begin{tikzpicture}[x=0.75pt,y=0.75pt,yscale=-1,xscale=1]

\draw   (464,30) -- (180,276) -- (464,276) -- cycle ;
\draw  [dash pattern={on 0.84pt off 2.51pt}]  (283,190) -- (283,273) ;
\draw [shift={(283,276)}, rotate = 270] [fill={rgb, 255:red, 0; green, 0; blue, 0 }  ][line width=0.08]  [draw opacity=0] (8.93,-4.29) -- (0,0) -- (8.93,4.29) -- cycle    ;
\draw [shift={(283,187)}, rotate = 90] [fill={rgb, 255:red, 0; green, 0; blue, 0 }  ][line width=0.08]  [draw opacity=0] (8.93,-4.29) -- (0,0) -- (8.93,4.29) -- cycle    ;
\draw  [fill={rgb, 255:red, 248; green, 231; blue, 28 }  ,fill opacity=1 ] (219,252) .. controls (219,241.51) and (232.43,233) .. (249,233) .. controls (265.57,233) and (279,241.51) .. (279,252) .. controls (279,262.49) and (265.57,271) .. (249,271) .. controls (232.43,271) and (219,262.49) .. (219,252) -- cycle ;

\draw (302,201.4) node [anchor=north west][inner sep=0.75pt]    {$\frac{CB}{\ln\frac{1}{B}}$};
\draw (240,241.4) node [anchor=north west][inner sep=0.75pt]    {$\Omega _{B}$};

\end{tikzpicture}

    \caption{Graph of $\Omega_{B}$}
\end{figure}
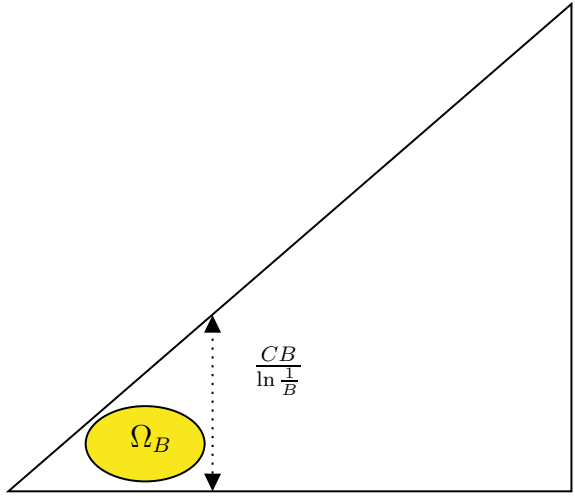
\begin{proof}
The set $\Omega_{B}$ in question is contained in \[\left\{(x,y): 0<y<x, \psi_0(x,y)<2By\right\},\]  since $\psi_0-\psi$ vanishes on $y=0$ and its derivative is bounded by $B.$ By Lemma \ref{PsiProperties}, it follows that 
\[x\ln\frac{1}{x}<C B.\] The result follows. 
\end{proof}

We now give the proof of Lemma \ref{level set estimate2}
\begin{proof}[Proof of Lemma \ref{level set estimate2}]
    By the various symmetry assumptions of $\psi\in\mathcal{K}_{\epsilon}$, it suffices to restrict the study on the domain \[\Omega=\left\{(x,y)\in\mathcal{T}:|\psi(x,y)|<2\epsilon\right\}.\]
    We divide the domain $\Omega$ into two parts:
   \begin{align}\label{Omega1}
    &\Omega_1=\left\{\psi_{0}(x,y)>2\psi(x,y)\right\}\cap \Omega,\\
    & \Omega_2=\left\{\psi_{0}(x,y)<2\psi(x,y)\right\}\cap\Omega.
\end{align}
It is obvious that \[\Omega_1 \subseteq \left\{(x,y)\in\mathcal{T}:(\psi_0-\psi)(x,y)>\frac{1}{2}\psi_0(x,y)\right\}.\] By applying Lemma \ref{small error}, with $B= C_1\sqrt{\epsilon\ln\frac{1}{\epsilon}}$, we have \begin{equation}\label{levelset5}
    \big |\Omega_1\big |\leq CC_1\frac{\epsilon \ln\frac{1}{\epsilon}}{\big(\ln\frac{1}{C_1}+\ln\dfrac{1}{\epsilon\ln\frac{1}{\epsilon}}\big)^2}\leq C C_1\frac{\epsilon}{\ln{\frac{1}{\epsilon}}}.
\end{equation}
Moreover, we have $\Omega_2\subseteq \left\{(x,y)\in\mathcal{T}:\psi_0<4\epsilon\right\}$, apply Lemma \ref{level set estimate2}, we have the following estimate of $\Omega_2$:
\begin{equation}\label{levelset6}
    \big |\Omega_2\big |\leq C \epsilon \ln\ln\frac{1}{\epsilon}.
\end{equation}
We thus close the proof of Lemma \ref{level set estimate2} by combining \eqref{levelset5} and \eqref{levelset6}.
\end{proof}
Steiner type estimate below will be used in the proof of Lemma \ref{lemma0}.
\begin{lemma}[Steiner type estimate]\label{Steiner}
$\int_{\Omega}\frac{1}{|(x,y)-(x_1,y_1)|}dx_1dy_1\leq C \sqrt{\big |\Omega\big |}$.
\end{lemma}
The proof of Lemma \ref{Steiner} simply follows from writing the integral in polar coordinates with respect to $(x,y)$.
\subsubsection{Proof of Lemma \ref{lemma0}}
 We now establish Lemma \ref{lemma0}. By the symmetries enjoyed by elements of $\mathcal{K}_\epsilon$, it suffices to prove that for all $(x,y)\in\mathcal{T}$, we have:
 \begin{equation*}
     \big |\nabla (T_{\epsilon}\psi-\psi_{0})(x,y)\big |\leq C_1 \sqrt{\epsilon\ln\frac{1}{\epsilon}}.
 \end{equation*}
Using the various symmetry assumptions on $\psi$ and Lemmas \ref{GreensFunction} and \ref{FProperties}, we have:
\begin{equation*}
    \begin{aligned}
        &\quad \big |\nabla (T_{\epsilon}\psi-\psi_{0})(x,y)\big |\\&=\left |\int_{\mathbb{T}^2}\big(F_{\epsilon}(\psi(x_1,y_1))+\text{sgn}(x_1)\text{sgn}(y_1)\big)\nabla_{x,y}G((x,y),(x_1,y_1))dx_1dy_1\right|\\&\leq C\int_{\mathbb{T}^{2+}}\big(1+\frac{1}{|(x-x_1,y-y_1)|}\big)\big |F_{\epsilon}(\psi(x_1,y_1))+1\big |dx_1dy_1\\&\leq \int_{\mathbb{T}^{2+}\cap\{\psi(x,y)<2\epsilon\}}\frac{C}{|(x-x_1,y-y_1)|}dx_1dy_1 .        
    \end{aligned}
\end{equation*}
Note that in the second equality above, we used that $F_\epsilon(z)= -1$ when $
z>2\epsilon.$
Furthermore, using Lemma \ref{Steiner}, Lemma \ref{level set estimate2},
we have
\begin{equation*}
    \begin{aligned}    &\quad \big |\nabla (T_{\epsilon}\psi-\psi_{0})(x,y)\big |\\&\leq \sqrt{\big |\left\{\psi(x,y)<2\epsilon\right\}\big |} \\&\leq C \sqrt{\epsilon \ln\ln\frac{1}{\epsilon}}+CC_1\sqrt{\frac{\epsilon}{\ln\frac{1}{\epsilon}}}\\ &\leq \frac{C_1\sqrt{\epsilon\ln\frac{1}{\epsilon}}}{4}.
    \end{aligned}
\end{equation*}
This concludes the proof of Lemma \ref{lemma0}.
\subsubsection{Return to the proof of Theorem \ref{step1} }
Now that we have Lemma \ref{lemma0}, by the Schauder fixed point theorem, we can find $\psi_{\epsilon}$ such that 
 \begin{align*}
     &\|\psi_{\epsilon}-\psi_0\|_{C^{1}(\mathbb{T}^2)}\leq C_1 \sqrt{\epsilon\ln\frac{1}{\epsilon}},\\
     &\psi_{\epsilon}=\triangle^{-1}F_{\epsilon}(\psi_{\epsilon}).
 \end{align*}  Moreover, since $F_\epsilon$ is smooth, $\psi_\epsilon$ is also smooth for any $\epsilon>0$ with the $\epsilon$-independent bound: 
 \begin{equation}\label{Holder endpoint}
     \|\triangle \psi_{\epsilon}\|_{L^{\infty}}=\|F_{\epsilon}(\psi_{\epsilon})\|_{L^{\infty}}\leq 1.
 \end{equation}
 From \eqref{Holder endpoint}, we have 
 for $\beta \in (\alpha,1)$, there is a constant $C_{\beta}$, such that \begin{equation*}
     \|\psi_{\epsilon}\|_{C^{1,\beta}(\mathbb{T}^2)}\leq C_{\beta}.
 \end{equation*}
Using interpolation, it is not difficult to show then that
 \begin{equation*}
    \lim_{\epsilon\rightarrow 0} \|\psi_{\epsilon}-\psi_0\|_{C^{1,\alpha}(\mathbb{T}^2)}=0.
\end{equation*}
This concludes the proof of Theorem \ref{step1}, which gives the existence of a family of smooth steady states $\psi_{\epsilon}$ near $\psi_0$. The proof does not give any information on the dependence of $\psi_\epsilon$ on $\epsilon$. We will consider this in the next subsection and thus complete the proof of Theorem \ref{main}.

\subsection{The smooth dependence of $\psi_{\epsilon}$ with respect to $\epsilon$}
\yupei{In this section, we will first recall a version of the inverse function theorem:\begin{theorem}[Inverse function theorem for steady state] \label{stability}
Let  $F(t,x) \in C^{\infty}(\mathbb{R}^2)$, we assume \begin{align*}
    &\triangle \psi_0^{*}=F(0,\psi_0^{*}),\\
    &\psi_0^{*} \in \mathcal{H}^2:=H^2(\mathbb{T}^2)\cap \left\{\psi|\psi(x,y)=\psi(-x,-y)=-\psi(-x,y)=-\psi(x,-y)=\psi(y,x)\right\}.
\end{align*} 
Furthermore, we assume the following:
\begin{itemize}
    \item [$(\mathbb{H}_1)$] $\triangle-D_{x}F(0,\psi^{*})$ be an isomorphism from $\mathcal{H}^2$ to
\begin{equation*}
    \mathcal{L}^2:= L^2(\mathbb{T}^2)\cap\left\{\psi|\psi(x,y)=\psi(-x,-y)=-\psi(-x,y)=-\psi(x,-y)=\psi(y,x)\right\}.
\end{equation*}
\end{itemize}
 Then there are positive constants $a$ and $b$, such that \begin{itemize}
     \item for $|t|<a,$ there is uniquely a function $\psi_t^{*}$ such that \begin{equation}
         \begin{aligned}
             \|\psi_t^{*}-\psi_{0}^{*}\|_{\mathcal{H}^2}<b,
         \end{aligned}
     \end{equation}
     \item $\psi_{(\cdot)}^{*}$ is a smooth map from $(-a,a)$ to $\mathcal{H}^2$.
 \end{itemize}
\end{theorem}
 We will
 prove the steady states we constructed in Theorem \ref{step1} satisfy the assumption $\mathbb{H}_1$ of Theorem \ref{stability}:\begin{theorem}\label{invertible}
 Let $\psi_{\epsilon}$ belong to \begin{equation}
    \mathcal{C}_{\epsilon}:=\left\{\psi|\psi(x,y)=\psi(y,x)=-\psi(x,-y)=-\psi(-x,y),\|\psi-\psi_0\|_{C^1(\mathbb{T}^2)}\leq C_1\sqrt{\epsilon\ln\frac{1}{\epsilon}}\ln\ln\frac{1}{\epsilon}\right\}
\end{equation} 
and satisfy the semi-linear elliptic equation \eqref{semi-linearequation}.
There is a positive constant $\epsilon_0$, such that for $0<\epsilon<\epsilon_0$,
 $\triangle-F^{'}_{\epsilon}(\psi_{\epsilon})$ is an isomorphism from $\mathcal{H}^2$
 to $\mathcal{L}^2$.
\end{theorem} 
Based on Theorem \ref{stability} and Theorem \ref{invertible}
, for every fixed $\psi_{\epsilon}$ we constructed in Theorem \ref{step1}, there is a smooth curve consisting of steady states satisfying \eqref{semi-linearequation} passing through $\psi_{\epsilon}$.
We can then prove the $\left\{\psi_{\epsilon}\right\}$ we construct in Theorem \ref{step1} is a smooth curve consisting of smooth steady states close to $\psi_0$ based on the following uniqueness result concerning the solutions of \eqref{semi-linearequation}\begin{theorem}\label{uniqueness}
There is a positive constant  $\epsilon_0$, such that for any fixed $\epsilon\in (0,\epsilon_0)$, the following statement holds. 
Assume that $\psi^{1}$ and $\psi^{2}$ are two solutions to the semi-linear elliptic equation \eqref{semi-linearequation} in the function space $\mathcal{C}_{\epsilon}$,
 then \begin{equation*}
    \psi^1=\psi^2.
\end{equation*}
\end{theorem}
}

\subsubsection{On the proof of Theorem \ref{invertible}}
In this section, we will discuss the proof of Theorem \ref{invertible} to verify that the steady states we construct in Theorem \ref{step1} satisfy the assumption of Theorem \ref{stability}.
Firstly, we need the lemma below which estimates the thickness of the sublevel set of $\psi\in \mathcal{C}_{\epsilon}$:
\begin{lemma}\label{levelset estimate3}
For any $\psi\in \mathcal{C}_{\epsilon}$, we have the estimate \begin{equation}
    \Omega=\left\{(x,y)\in \mathcal{T}:|\psi|<2\epsilon\right\}\subseteq \Omega_{level}\cup \Omega_{error},
\end{equation}
    where $\Omega_{level}$ and $\Omega_{error}$ is defined as \begin{equation}\label{Omega5}
        \Omega_{level}:=\left\{(x,y):y\leq\frac{C\sqrt{\epsilon}}{\ln\frac{1}{\epsilon}}\right\},
    \end{equation}
    \begin{equation}\label{Omega6}
        \Omega_{error}:=\left\{(x,y):0<y<x<CC_1\frac{\sqrt{\epsilon}}{(\ln{\frac{1}{\epsilon}})^{\frac{1}{4}}}\right\}.
    \end{equation}
\end{lemma}
We now give the proof of Lemma \ref{levelset estimate3}
\begin{proof}[Proof of Lemma \ref{levelset estimate3}]
    Let $\Omega_1, \Omega_2$ be the ones defined as below: \begin{align}
    &\Omega_1=\left\{\psi_{0}(x,y)>2\psi(x,y)\right\}\cap \Omega,\\
    & \Omega_2=\left\{\psi_{0}(x,y)<2\psi(x,y)\right\}\cap\Omega.
\end{align}
    we have \begin{equation}
        \Omega\subseteq \Omega_1 \cup \Omega_2.
    \end{equation}
   By Lemma \ref{small error} with $B= C_1\sqrt{\epsilon\ln\frac{1}{\epsilon}}\ln\ln\frac{1}{\epsilon}$, we have
    $\Omega_1\subseteq \Omega_{error}$.
     In $\Omega_2$, we use the bound \begin{equation}
        y<Cmin(x,\frac{2\epsilon}{x\ln x}),
    \end{equation} and we have $\Omega_2\subseteq \Omega_{level}$.\\ The proof is now completed.
    \begin{figure}[h]
        \centering
\begin{minipage}[t]{.5\textwidth}
\centering

\tikzset{every picture/.style={line width=0.75pt}} 

\begin{tikzpicture}[x=0.75pt,y=0.75pt,yscale=-0.75,xscale=0.75]

\draw   (471,59) -- (235,262) -- (471,262) -- cycle ;
\draw [color={rgb, 255:red, 0; green, 0; blue, 0 }  ,draw opacity=1 ] [dash pattern={on 0.84pt off 2.51pt}]  (339.03,176) -- (339.97,262) ;
\draw [shift={(340,265)}, rotate = 269.38] [fill={rgb, 255:red, 0; green, 0; blue, 0 }  ,fill opacity=1 ][line width=0.08]  [draw opacity=0] (10.72,-5.15) -- (0,0) -- (10.72,5.15) -- (7.12,0) -- cycle    ;
\draw [shift={(339,173)}, rotate = 89.38] [fill={rgb, 255:red, 0; green, 0; blue, 0 }  ,fill opacity=1 ][line width=0.08]  [draw opacity=0] (8.93,-4.29) -- (0,0) -- (8.93,4.29) -- cycle    ;
\draw  [fill={rgb, 255:red, 80; green, 227; blue, 194 }  ,fill opacity=1 ] (299,218) .. controls (315,194) and (322,220) .. (325,236) .. controls (328,252) and (317,248) .. (307,254) .. controls (297,260) and (301,250) .. (291,246) .. controls (281,242) and (283,242) .. (299,218) -- cycle ;

\draw (296,227.4) node [anchor=north west][inner sep=0.75pt]    {$\Omega _{1}$};
\draw (348,193.4) node [anchor=north west][inner sep=0.75pt]  [font=\tiny]  {$CC_{1}\sqrt{\epsilon } \ \left(\ln\frac{1}{\epsilon }\right)^{\frac{-1}{4}}$};

\end{tikzpicture}

\caption{Graph of $\Omega_1$}
    \end{minipage}
      \begin{minipage}[t]{.45\textwidth}    

\tikzset{every picture/.style={line width=0.75pt}} 

\begin{tikzpicture}[x=0.75pt,y=0.75pt,yscale=-0.5,xscale=0.5]

\draw   (504,11) -- (189,285) -- (504,285) -- cycle ;
\draw    (269,217) -- (503,217) ;
\draw  [fill={rgb, 255:red, 80; green, 227; blue, 194 }  ,fill opacity=1 ] (240,244) .. controls (248.43,237.15) and (307.24,226.61) .. (359.13,234.57) .. controls (411.03,242.54) and (362.31,235.07) .. (363.89,235.32) .. controls (365.48,235.58) and (446.49,243.84) .. (483.32,255.63) .. controls (520.15,267.43) and (498,266) .. (501,266) .. controls (504,266) and (505.16,260.96) .. (505,276) .. controls (504.84,291.04) and (494.43,282.61) .. (489,284) .. controls (483.57,285.39) and (387.57,284.33) .. (331,285) .. controls (274.43,285.67) and (199,284) .. (189,285) .. controls (179,286) and (231.57,250.85) .. (240,244) -- cycle ;
\draw [fill={rgb, 255:red, 139; green, 87; blue, 42 }  ,fill opacity=1 ] [dash pattern={on 0.84pt off 2.51pt}]  (395.09,221) -- (396.91,285) ;
\draw [shift={(397,288)}, rotate = 268.36] [fill={rgb, 255:red, 0; green, 0; blue, 0 }  ][line width=0.08]  [draw opacity=0] (10.72,-5.15) -- (0,0) -- (10.72,5.15) -- (7.12,0) -- cycle    ;
\draw [shift={(395,218)}, rotate = 88.36] [fill={rgb, 255:red, 0; green, 0; blue, 0 }  ][line width=0.08]  [draw opacity=0] (8.93,-4.29) -- (0,0) -- (8.93,4.29) -- cycle    ;

\draw (412,225.4) node [anchor=north west][inner sep=0.75pt]  [font=\footnotesize]  {$\frac{C\sqrt{\epsilon }}{\ln\left(\frac{1}{\epsilon }\right)}$};
\draw (345,251.4) node [anchor=north west][inner sep=0.75pt]    {$\Omega _{2}$};

\end{tikzpicture}

        \caption{Graph of $\Omega_2$}
\end{minipage}%
    \end{figure}
\end{proof}
We now to apply Lemma \ref{levelset estimate3} to prove Theorem \ref{invertible}.
\begin{proof}[Proof of Theorem \ref{invertible}]
Since $\triangle$ is  an isomorphism from  $\mathcal{H}^2$ to $\mathcal{L}^2$, and $F_{\epsilon}^{'}$ is relative compact to $\Delta$, by the Fredholm alternative, $\triangle-F_{\epsilon}^{'}$ is isomorphism if and only if it is injective.
\\Per absurdum, if $\triangle-F_{\epsilon}^{'}$ is not isomorphism, there is non-zero $\psi \in \mathcal{H}^2$  such that \begin{equation}\label{absurdum}   \triangle\psi=F_{\epsilon}^{'}(\psi_{\epsilon}) \psi.
\end{equation} 
Multiply \eqref{absurdum} by $\psi$ and use integration by parts, from various symmetry assumptions on $F_{\epsilon}$, $\psi$ and $\psi_{\epsilon}$ we have:
\begin{equation}
    \begin{aligned}
&\quad\int_{\mathbb{T}^2}\big |\nabla \psi \big |^2dxdy\\&=16\int_{\mathcal{T}} F_{\epsilon}^{'}(\psi_{\epsilon}) \big |\psi\big |^2dxdy     
    \end{aligned}
\end{equation}
From Lemma \ref{FProperties} and Lemma \ref{levelset estimate3}, we have
\begin{equation}\label{psiunique1}
    \begin{aligned}
        &\quad\int_{\mathbb{T}^2}\big |\nabla \psi \big |^2dxdy \leq \frac{C}{\epsilon}\int_{\{\psi_{\epsilon}<2\epsilon\}\cap\mathcal{T}}  \big |\psi\big |^2dxdy\\&\leq \frac{C}{\epsilon}\int_{\Omega_{level}}\big |\psi\big |^2dxdy+\frac{C}{\epsilon}\int_{\Omega_{error}}\big |\psi\big |^2dxdy:=J_{1}^1+J_{1}^2.
    \end{aligned}
\end{equation}.
\\\noindent\emph{Estimate of $J_1^{1}$:}\\
From H\H older inequality, we have
\begin{equation}\label{mid1theorem2.4}
    \begin{aligned}
\big |\psi(x,y)\big |^2&=\big |\psi(x,0)+\psi(x,y)-\psi(x,0)\big |^2\leq \left(\int_{0}^{y} \big |\psi_{y}\big |(x,s)ds\right)^2\\&\leq y \int_{0}^{y}\big |\psi_{y}\big |^2(x,s)ds\leq y \int_{0}^{\frac{1}{2}}\big |\psi_{y}\big |^2(x,s)ds.        
    \end{aligned}
\end{equation}
By \eqref{mid1theorem2.4} and \eqref{Omega5}, we have 
\begin{equation}\label{psiunique2}
    \begin{aligned}
        J_{1}^{1}&\leq\frac{C}{\ln^2(\epsilon)}\int_{\mathcal{T}} \big |\psi_{y}\big |^2dxds\\&\leq\frac{1}{16} \int_{\mathbb{T}^2}\big |\nabla \psi\big |^2dxdy.
    \end{aligned}
\end{equation}
\noindent\emph{Estimate on $J_{1}^{2}$:}\\
 By \eqref{mid1theorem2.4} and \eqref{Omega6}, we have
\begin{equation}\label{psiunique3}
\begin{aligned}
    J_{1}^2&\leq \frac{C}{\sqrt{\ln\frac{1}{\epsilon}}}\int_{\mathcal{T}}\big |\nabla \psi\big |^2dxdy\\&\leq\frac{1}{16}\int_{\mathbb{T}^2}\big |\nabla \psi\big |^2dxdy.
    \end{aligned}
\end{equation}
 By \eqref{psiunique1}, \eqref{psiunique2} and \eqref{psiunique3},
 we would have
 \begin{equation*}
     \int_{\mathbb{T}^2}\big |\nabla \psi\big |^2dxdy\leq \frac{1}{2}\int_{\mathbb{T}^2}\big |\nabla \psi\big |^2dxdy,
 \end{equation*} which implies
 \begin{equation*}
     \nabla\psi=0.
 \end{equation*}  The result follows.
\end{proof}
\subsubsection{The sketch of proof of Theorem \ref{uniqueness}}
The proof of Theorem \ref{uniqueness} is similar to the proof of Theorem \ref{invertible} and we will sketch here.
\begin{proof}[Sketch of the proof of Theorem \ref{uniqueness}]
Since $\Delta(\psi^1-\psi^2)=F_{\epsilon}(\psi^1)-F_{\epsilon}(\psi^2)$, by various symmetry assumptions on $F_{\epsilon}$, $\psi^1$ and $\psi^2$, we have 
\begin{equation}
    \begin{aligned}\label{psiunique4}
        &\quad\int_{\mathbb{T}^2} \big |\nabla(\psi^1-\psi^2)\big |^2dxdy\\&=16\int_{\mathcal{T}} (\psi_1-\psi_2)(F_{\epsilon}(\psi^2)-F_{\epsilon}(\psi^1))dxdy\\&\leq\frac{C}{\epsilon}\int_{\mathcal{T}\cap\{\psi^1(x,y)<4\epsilon\}} \big |\psi^1-\psi^2\big |^2dxdy+\frac{C}{\epsilon}\int_{\mathcal{T}\cap\{\psi^2(x,y)<4\epsilon\}}\big |\psi^1-\psi^2\big |^2dxdy.
    \end{aligned}
\end{equation}
Since $\psi^1,\psi^2\in\mathcal{C}_{\epsilon}$
from Lemma \ref{levelset estimate3}, we have \begin{equation}
\begin{aligned}
     &\mathcal{T}\cap\left\{\psi^{i}(x,y)<4\epsilon\right\} \subseteq \Omega_{level} \cup \Omega_{error},\text{for $i=1,2,$}
\end{aligned}
\end{equation}
where $\Omega_{level}$ is defined in \eqref{Omega5} and $\Omega_{error}$ is defined in \eqref{Omega6}.
Again by \eqref{mid1theorem2.4}, from \eqref{psiunique4}
 we have:
\begin{equation*}
    \begin{aligned}
&\quad\int_{\mathbb{T}^2}\big |\nabla(\psi^1-\psi^2)\big |^2dxdy\\&\leq \frac{C}{\sqrt{\ln\frac{1}{\epsilon}}}\int_{\mathbb{T}^2}\big |\nabla(\psi^1-\psi^2)\big |^2dxdy\\&\leq\frac{1}{2}\int_{\mathbb{T}^2}\big |\nabla(\psi^1-\psi^2)\big |^2dxdy,   
    \end{aligned}
\end{equation*}
which implies $\psi^1=\psi^2,$ and we get a contradiction.
\end{proof}
\subsubsection{Concluding the proof of Theorem \ref{main}} 
Given the setup of Theorem \ref{main}, by Theorem \ref{invertible}, and Theorem \ref{stability}  we have: $\forall \epsilon^{*}, 0<\epsilon^{*}<\epsilon_0$, we can find constant $a(\epsilon^*)$ and a family of function $\left\{\psi^{*}_{\epsilon}\right\}$, such that \begin{equation}\label{smooth1}
    \psi^{*}_{.}(.,.)\in C^{\infty}((\epsilon^*-a,\epsilon^*+a)\times \mathbb{T}^2),
\end{equation} and \begin{equation}\label{smooth2}
    \begin{aligned}
    &\psi^{*}_{\epsilon^{*}}=\psi_{\epsilon^{*}},\\
    &\psi^*_{\epsilon}=\triangle^{-1}F_{\epsilon}(\psi^{*}_{\epsilon}),\\ &\|\psi^*_{\epsilon}-\psi_0\|_{C^{1}(\mathbb{T}^2)}<C_1\sqrt{\epsilon}\sqrt{ln\frac{1}{\epsilon}}\ln\ln\frac{1}{\epsilon}.
    \end{aligned}
\end{equation} 
While in the Theorem \ref{step1}, the steady states we construct satisfy the following equation: \begin{equation}
    \begin{aligned}
    &\psi_{\epsilon}=\triangle^{-1}F_{\epsilon}(\psi_{\epsilon}),\\ &\|\psi_{\epsilon}-\psi_0\|_{C^{1}(\mathbb{T}^2)}<C_1 \sqrt{\epsilon\ln\frac{1}{\epsilon}}<C_1\sqrt{\epsilon\ln\frac{1}{\epsilon}}\ln\ln\frac{1}{\epsilon}.
    \end{aligned}
\end{equation}  
By Theorem \ref{uniqueness} and \eqref{smooth2}, we proved \begin{equation}\label{extension}
    \psi^*_{\epsilon}=\psi_{\epsilon}.
\end{equation} 
The equation \eqref{extension} in particular implies that the steady states $\left\{\psi_{\epsilon}\right\}$ we constructed in Theorem \ref{step1} is a smooth curve consisting of smooth steady states that converges to $\psi_0$ as $\epsilon$ tends to $0$. 
\section{Singular steady states near $\psi_0$}
In this section, we will discuss 
  solutions of
  \begin{equation}\label{algebraic}
  \Delta \phi_{\epsilon}=P_{\epsilon}(\phi_{\epsilon})    
  \end{equation}
   in $\mathbb{T}^2$, where $P_{\epsilon}(x)$ is an odd function such that 
\begin{align*}
    P_{\epsilon}(x)=&\left\{\begin{array}{cc}
        - \frac{\epsilon^s}{x^s}, &\quad 0<x<\epsilon,\\
         -1, &\quad x\geq \epsilon.
    \end{array}\right.
\end{align*}
We will prove the following statement:
\begin{theorem}\label{Algebraic singularity}
Let $0<s<1$, there is a constant $\epsilon_0>0$, if 
   $0<\epsilon<\epsilon_0,$
 \eqref{algebraic} has a unique $ C^1(\mathbb{T}^2)$ solution $\phi_{\epsilon}$ in the class 
    \begin{equation}
        \begin{aligned}
            &\mathcal{D}:=\left\{\phi|\phi(x,y)=-\phi(-x,y)=-\phi(x,-y)=\phi(y,x)\right\}\\&\cap \left \{\phi|\phi(x,y)\geq \psi_0(x,y), \text{for $(x,y)\in\mathcal{T}$}\right\}.
        \end{aligned}
    \end{equation}
 Moreover, we have the following property \begin{itemize}
    \item \begin{equation}
    \lim_{\epsilon\rightarrow 0}\|\phi_{\epsilon}-\psi_0\|_{C^{1}(\mathbb{T}^2)}=0,
\end{equation}
\item The Lagrangian flow generated by $\phi_\epsilon$ becomes immediately discontinuous as multiple trajectories can emanate from or arrive at the origin in finite time.  
\end{itemize}   
\end{theorem}
The uniqueness of the solution can be derived from the maximum principle. (see for example, \cite{Pino1992AGE}.) For the existence of solutions and the $C^1$ convergence, when $0<s<\frac{1}{2}$, we could use similar arguments as what we did in the proof of Theorem \ref{main}. For the case $\frac{1}{2}\leq s<1,$ we need better control on $\phi_\epsilon$ near the origin.
\subsection{On the proof of Theorem \ref{Algebraic singularity}}
We will first state and prove a uniform estimate for $P_{\epsilon}(\phi_{\epsilon})$.
\begin{lemma}\label{Uniform upper bound}
Let $1\leq p<\frac{1}{s}$, for any $\phi_{\epsilon}\in \mathcal{D}$, there is a positive constant $C(p)<\infty$, such that if  $\epsilon<\epsilon_0$ we have \begin{equation}
    \|P_{\epsilon}(\phi_{\epsilon})\|_{L^{p}(\mathbb{T}^2)}\leq C(p).
\end{equation}
In particular, by H\H older inequality, \begin{equation}
    P_{\epsilon}(\phi_{\epsilon})\in L^1.
\end{equation}
\end{lemma} 
\begin{proof}[Proof of Lemma \ref{Uniform upper bound}]
By the construction of $P_{\epsilon}$, 
we have \begin{equation*}
    \big |P_{\epsilon}(\phi_{\epsilon})\big |\leq 1+ \frac{\epsilon^{s}}{\big |\phi_{\epsilon}\big |^s}\chi_{|\psi_0|<\epsilon}.
    \end{equation*}
Since $\big |\phi_{\epsilon}\big |>\big |\psi_{0}\big |,$ we have
\begin{equation*}
     \big |P_{\epsilon}(\phi_{\epsilon})\big |\leq 1+ \frac{\epsilon^{s}}{|\psi_{0}|^s}\chi_{|\psi_0|<\epsilon}.
\end{equation*}
Then by Lemma \ref{PsiProperties}, we have \begin{equation}\label{upper bound}
    \big |P_{\epsilon}(\phi_{\epsilon})(x,y)\big |\leq  C(1+\frac{1}{|xy\ln{(x^2+y^2)}|^{ps}}).
\end{equation}
\eqref{upper bound} finishes the proof of Lemma \ref{Uniform upper bound}.
\end{proof}
We now divide the proof of Theorem \ref{Algebraic singularity} into three parts:
\begin{itemize}
    \item{1} The existence of the solution in the Theorem \ref{Algebraic singularity} when $0<s<\frac{1}{2}$.
    \item{2} The existence of the solution in the Theorem \ref{Algebraic singularity} when $\frac{1}{2}\leq s<1$.
    \item{3} Existence of Lagrangian trajectories emanating from the origin.
\end{itemize}

\subsubsection{On the existence of the solution in Theorem \ref{Algebraic singularity} when $0<s<\frac{1}{2}$}
In this section, we prove the existence of solutions to \eqref{algebraic} when $0<s<\frac{1}{2}$.
We have the following result: 
\begin{lemma}\label{lemma2}
Let $0<s<\frac{1}{2}$. Define $T_\epsilon$ by
\[T_{\epsilon}(\phi)=\triangle^{-1}P_{\epsilon}(\phi).\] Then, there exists
$\epsilon_0>0$, $C(s)>0$, so that for all $ 0<\epsilon <\epsilon_0$, the set \[\left\{\psi|\|\phi-\psi_0\|_{C^{1}(\mathbb{T}^2)}\leq C(s)\sqrt{-\epsilon \ln\epsilon}\right\}\cap \mathcal{D}\] is invariant under the mapping $T_{\epsilon}$.
\end{lemma}
\begin{proof}[Proof of Lemma \ref{lemma2}]
By the explicit form of $T_{\epsilon}$, we can prove the various symmetries given in the definition of $\mathcal{D}$ are persevered under $T_{\epsilon}$. Then from maximum principle, we get $\phi\geq \psi_{0}$ in $\mathcal{T}$ is also perserved.
Moreover,by Lemma \ref{GreensFunction}, and using the various symmetry assumption on $P_{\epsilon}$,$\phi$ and $\psi_0$, similar to the proof of Theorem \ref{main}, we have 
\begin{equation}\label{exi1}
\begin{aligned}
    &\quad\big |\nabla(T_{\epsilon}\phi-\psi_0)(x,y)\big |\\&=\left|\int_{\mathbb{T}^2} \nabla_{x,y}G((x,y),(x_1,y_1))(P_{\epsilon}(\phi)-\Delta \psi_0)(x_1,y_1)dx_1dy_1\right|\\&\leq C
 \int_{\mathbb{T}^{2+}} \frac{C}{|(x_1-x,y-y_1)|}\big |P_{\epsilon}(\phi)(x_1,y_1)+1\big |dx_1dy_1
\end{aligned}    
\end{equation}
  By Lemma \ref{PsiProperties}  Lemma \ref{level set} and \eqref{exi1} we have
\begin{equation}\label{exi2}
    \begin{aligned}
        &\quad \big |\nabla(T_{\epsilon}\phi-\psi_0)(x,y)\big |\\&=\sum_{k=1}^{\infty}\int_{\mathbb{T}^{2+}\cap\{\frac{\epsilon}{2^{k-1}}>\phi(x_1.y_1)>\frac{\epsilon}{2^{k}}\}}\frac{C}{|(x_1-x,y-y_1)|}\big |P_{\epsilon}(\phi)(x_1,y_1)+1\big |dx_1dy_1\\&\leq \sum_{k=1}^{\infty}\int_{\mathbb{T}^{2+}\cap\{\frac{\epsilon}{2^{k-1}}>\phi(x_1.y_1)\}}\frac{C}{|(x_1-x,y-y_1)|}|\times [2^{ks}+1] dx_1dy_1\\&\leq C \sum_{k=1}^{\infty} \frac{2^{ks}+1}{2^{\frac{k}{2}}}\sqrt{\epsilon \ln\frac{2^k}{\epsilon}}\\&\leq C\sqrt{\epsilon \ln\frac{1}{\epsilon}}.
    \end{aligned}
\end{equation}
 The proof of Lemma \ref{lemma2} is now finished.
\end{proof}
 \begin{remark}
     By Interpolation theorem in H\H older space, it is not difficult to show as $\epsilon \rightarrow 0$, for $0<\alpha<1-2s$
\begin{equation}
    \|\phi_{\epsilon}-\psi_0\|_{C^{1,\alpha}(\mathbb{T}^2)}\rightarrow 0.
\end{equation}
 \end{remark} 
 In the above we finish the proof of Theorem \ref{Algebraic singularity} in the case where $0<s<\frac{1}{2}$.
In the following section, we will discuss the proof of Theorem \ref{Algebraic singularity} when the degree of algebraic singularity s satisfies $\frac{1}{2}\leq s< 1$.
\subsection{On the proof of Theorem \ref{Algebraic singularity} when $\frac{1}{2}\leq s<1$}
In this section, we will use an argument similar to the previous section to construct a solution to \eqref{algebraic} in low regularity. We will then prove bounds on the solution in a neighborhood of the origin in the first quadrant. Based on this fact, we can use the method of \cite{Pino1992AGE} to prove the solution is $C^1$ and that it converges to $\phi_0$ in $C^1(\mathbb{T}^2)$ as $\epsilon\rightarrow 0$.\\
\subsubsection{On the existence of solutions to \eqref{algebraic} in low regularity}
The lemma below proves the existence of a solution to \eqref{algebraic} in low regularity.
\begin{lemma}\label{lowerregularity existence}
For $\frac{1}{2}\leq s<1$, and  $q<\frac{1}{s}$, we have an solution $\phi_{\epsilon}\in W^{2,q}\cap \mathcal{D}$ for \eqref{algebraic}.
\end{lemma}
The proof of Lemma \ref{lowerregularity existence} is similar as what we did in the previous section and we will sketch here.
\begin{proof}[Sketch of the proof of Lemma \ref{lowerregularity existence}],
Define $P_{\epsilon}^{n}$ be an odd function such that  \begin{align*}
    P_{\epsilon}^{n}(x)=&\left \{\begin{array}{cc}
         -2^{ns},& \quad 0<x<\frac{\epsilon}{2^n},\\
    -\frac{\epsilon^s}{x^s},&\quad \frac{\epsilon}{2^n}<x<\epsilon,\\ -1,&\quad x\geq \epsilon.
    \end{array}\right. 
\end{align*}
We define the operator $T_{\epsilon}^{n}(\phi)=\Delta^{-1}P_{\epsilon}^{n}(\phi)$, 
similarly as what we did in Lemma \ref{lemma2}, for any $n$, we can show there is a $\phi_{\epsilon}^{n} \in \mathcal{D}$ such that \begin{equation}
    \phi_{\epsilon}^{n}=\Delta^{-1}[P_{\epsilon}^{n}(\phi_{\epsilon}^{n})].
\end{equation}  
Moreover, similarly, as what we did in the proof of Lemma \ref{Uniform upper bound}, we have 
\begin{equation}\label{upper bound2}
     \big |P_{\epsilon}^{n}(\phi_{\epsilon})(x,y)\big |\leq  C(1+\frac{1}{|xy|^{ps}}).
\end{equation}
As a consequence,
\begin{equation}
    \|\phi_{\epsilon}^{n}\|_{W^{2,p_0}}\leq C, \text{for $q<p_0<\frac{1}{s}$}.
\end{equation}
 From the Soblev embedding theorem, there is a subsequence $\left\{\phi_{\epsilon}^{n_k}\right\}$ of $\left\{\phi_{\epsilon}^{n}\right\}$ converges to $\phi_{\epsilon}$  almost everywhere.
 Moreover, by \eqref{upper bound2}, and dominated convergence Theorem, we have $\left\{P_{\epsilon}^{n_k}(\phi_{\epsilon}^{n_k})\right\}$ converges to $P_{\epsilon}(\phi_{\epsilon})$ in $L^{q}$.
As a consequence, we have
$\phi_{\epsilon}^{n_k}=\Delta^{-1}P_{\epsilon}^{n_k}(\phi_{\epsilon}^{n_k})$ converges to $\phi_{\epsilon}$ in $W^{2,q}$,
 and $\phi_{\epsilon}$ is a solution of \eqref{algebraic} in $\mathcal{D}\cap W^{2,q}$.
\end{proof}
\subsubsection{On the $C^1$ regularity of $\phi_{\epsilon}$}
In this section, our goal is to prove:
\begin{theorem}
\label{C1regularitytheorem}
    $\phi_{\epsilon}\in C^1(\mathbb{T}^2)$.
\end{theorem}
 We firstly get a lower bound for $\phi_{\epsilon}$ in a small neighborhood near the origin in the first quadrant:
\begin{lemma}\label{maximum principle}
There is a constant $C(s)$ so that for all $\epsilon$ sufficiently small and $(x,y)\in\Omega_3:=\mathbb{T}^{2+}\cap B_{C(s)\epsilon^{\frac{1}{2}}}$, we have \begin{equation*}
    \phi_{\epsilon}(x,y)\geq C\epsilon^{\frac{s}{s+1}}r^{\frac{2}{s+1}} \sin(2\theta),
\end{equation*}
where we write $(x,y)=(r\cos(\theta),r\sin(\theta))$.
\end{lemma}
To prove Lemma \ref{maximum principle}, we construct a barrier function:
\begin{lemma}\label{Barrier function}
    There is a solution $\tilde{\psi}$ to \begin{equation}\label{special}
\begin{aligned}
     &\Delta\tilde{\psi}=-\frac{\epsilon^s}{\tilde \psi^{s}},\text{for $(x,y)\in \mathbb{R}^{2+}:=\mathbb{R}^{+}\times \mathbb{R}^{+}$,}\\
     &\tilde{\psi}(x,y)=0,\text{for $(x,y)\in \partial\mathbb{R}^{2+}$,}
\end{aligned}
\end{equation}
such that \begin{equation}\label{maxiun principle 18}
    C\epsilon^{\frac{s}{s+1}}r^{\frac{2}{s+1}}\sin(2\theta)>\tilde{\psi}>\frac{1}{C}\epsilon^{\frac{s}{s+1}}r^{\frac{2}{s+1}}\sin(2\theta),\text{when $\theta \in[0,\frac{\pi}{2}]$.}
\end{equation} 
\end{lemma}
We will give the proof of Lemma \ref{Barrier function} in the Appendix.
Now we will present the proof for Lemma \ref{maximum principle} with the aid of Lemma \ref{Barrier function}.
\begin{proof}[Proof of Lemma \ref{maximum principle}]
By \eqref{maxiun principle 18}, we can choose $C(s)$ such that in $\Omega_3:=B_{C(s)\epsilon^{\frac{1}{2}}}\cap\mathbb{T}^{2+}$, we have \begin{equation}\label{lowerbound1}
    \tilde{\psi}<\epsilon.
\end{equation} 
Moreover in $\partial \Omega_3$, we have 
\begin{equation*}
     \phi_{\epsilon}>\psi_0\geq \frac{1}{C} \epsilon \ln\frac{1}{\epsilon} \sin(2\theta).
 \end{equation*} 
 While from \eqref{maxiun principle 18},
 \begin{equation*}
     \tilde{\psi}\leq C\epsilon sin(2\theta).
 \end{equation*}
 As a result, we have on $\partial \Omega_3 \cap \mathbb{T}^{2+}$,
 \begin{equation}\label{maximum principle 19}
    \phi_{\epsilon}\geq \tilde{\psi}.  
 \end{equation}
Now if $\tilde{\psi}-\phi_{\epsilon}$ achieves the positive maximum  at $(x_0,y_0))$ in $\Omega_3$,  by \eqref{maximum principle 19},  \begin{equation*}
     (x_0,y_0) \in Int(\Omega_3),
 \end{equation*}
 and \begin{equation}\label{maximum principle 2}
     \Delta(\tilde{\psi}-\phi_{\epsilon})(x_0,y_0)\leq 0.
 \end{equation} However, we have  $0<\phi_{\epsilon}(x_0,y_0)<\tilde{\psi}(x_0,y_0)<\epsilon,$ then \begin{equation*}
     \Delta \phi_{\epsilon}(x_0,y_0)=\frac{\epsilon^s}{(-\phi_{\epsilon})^s}(x_0,y_0)< \frac{\epsilon^s}{(-\tilde{\psi}^s)}(x_0,y_0)=\Delta \tilde{\psi}(x_0,y_0),
 \end{equation*} which leads to contradiction to \eqref{maximum principle 2} . \\
 We now see that in $\Omega_3$ we have \begin{equation*}
    \phi_{\epsilon}\geq \tilde{\psi},
 \end{equation*}  which completes the proof in view of \eqref{maxiun principle 18}.
 \end{proof}
 Now we are ready to prove Lemma \ref{C1regularitytheorem}. We will use a version of Lemma 3.3 in \cite{Pino1992AGE}:
 \begin{corollary} \label{Del}
  Let $0<R<2$, assume that $f\in L^p{(-R,R)}$, for some $p>1$. Let
  \begin{equation*}
      g(x,y)=\int_{B_R} \frac{1}{|(x,y)-(x_1,y_1)|}f(y_1)dx_1dy_1,
  \end{equation*} then there is a a constant $c(p)$ so that \begin{equation*}
      |g(x,y)|\leq c(p) R[\int_{-1}^{1}\big |f(Ry_1)\big |^{p}dy_1]^{\frac{1}{p}}.
  \end{equation*}
 \end{corollary}
 \begin{proof}[Proof of Theorem \ref{C1regularitytheorem}]
     From the trivial lower bound $\phi_\epsilon\geq \psi_0$ (in the first quadrant), we see that 
     \[\big |\Delta\phi_\epsilon\big |\leq \min\left\{\frac{\epsilon^s}{\psi_0^s},1\right\}.
     \] Now using Corollary \ref{Del} as well as the Green's function representation of $\phi_\epsilon$, we see that \begin{equation}\label{Cor3.1 1}
       \phi_{\epsilon} \in C^1 (\mathbb{T}^{2} - \{(0,0)\} ).
   \end{equation}
   Next, we will  establish the $C^1$ regularity on the whole of $\mathbb{T}^2$ using the more precise lower bound in Lemma \ref{maximum principle}.
   Let $r_0<C(s) \epsilon^{\frac{1}{2}}$, then for all points $(x,y) \in B_{r_0}$, we have that
  \begin{equation*}
 \begin{aligned}
     &\quad \nabla(\phi_{\epsilon}-\psi_0)(x,y)\\&=\int_{B_{r_0}}\frac{(x-x_1,y-y_1)}{|(x-x_1,y-y_1)|^2}\bigg(P_{\epsilon}(\phi_{\epsilon}(x_1,y_1))+\text{sgn}(x_1)\text{sgn}(y_1)\bigg)dx_1dy_1\\&+\int_{B_{r_0}^{c}\cap \mathbb{T}^2}\bigg(\frac{(x-x_1,y-y_1)}{|(x-x_1,y-y_1)|^2}-\frac{(-x_1,y_1)}{|(x_1,y_1)|^2}\bigg)\bigg(P_{\epsilon}(\phi_{\epsilon})(x_1,y_1)+\text{sgn}(x_1)\text{sgn}(y_1)\bigg)dx_1dy_1\\&+\int_{\mathbb{T}^2}\bigg(\nabla G_{x,y}((x,y),(x_1,y_1))-\frac{(x-x_1,y-y_1)}{|(x-x_1,y-y_1)|^2}\bigg)\bigg(P_{\epsilon}(\phi_{\epsilon})(x_1,y_1)+\text{sgn}(x_1)\text{sgn}(y_1))\bigg)dx_1dy_1.
 \end{aligned}    
 \end{equation*} 
 As a consequence, we have \begin{equation}\label{velocity estimate 2}
 \begin{aligned}
  &\quad\big |\nabla(\phi_{\epsilon}-\psi_0)(x,y)\big |\\&\leq \int_{B_{r_0}}\frac{C}{|(x-x_1,y-y_1)|}\bigg(1+\big |P_{\epsilon}(\phi_{\epsilon}) \big |\bigg)dx_1y_1\\& +\frac{1000|(x,y)|}{r_0^2}\int_{\mathbb{T}^2}\big |P_{\epsilon}(\phi_{\epsilon}(x_1,y_1))+\text{sgn}(x_1)\text{sgn}(y_1)\big |dx_1dy_1\\&:=I_1+I_2.
 \end{aligned}
 \end{equation}
 
 \noindent\emph{Estimate of $I_1$:}
 
 By the various symmetry properties of $\phi_{\epsilon}$, we have that
 \begin{equation}\label{Cor 3.1 3}
 \begin{aligned}
    &I_1(r_0,x,y)\leq \int_{B_{r_0}\cap\mathbb{T}^{2+}}\frac{C}{|(x-x_1,y-y_1)|}\bigg(1+|P_{\epsilon}(\phi_{\epsilon}) |\bigg)(x_1,y_1)dx_1y_1.
     \end{aligned}
 \end{equation}
Since in $B_{r_0}\cap\mathbb{T}^{2+}$, by Lemma \ref{maximum principle}, we would have \begin{equation}\label{Cor 3.1 2}
        \big |P(\phi_{\epsilon}(x,y))\big |\leq \frac{\epsilon^s}{|\phi|_{\epsilon}^s(x,y)}\leq \epsilon^{s}\frac{1}{\tilde{\psi}^s(x,y)}\leq C \epsilon^{\frac{s}{1+s}} \frac{1}{y^{\frac{2s}{s+1}}}.
\end{equation}  
By \eqref{Cor 3.1 2}, symmetry of $\phi_{\epsilon}$,  and apply Corollary \ref{Del}, we have:
\begin{equation}
    \big |I_1(x,y,r_0)\big |\leq \int_{B_{r_0}}\frac{C}{|(x_1-x,y_1-y)|}\bigg(1+\epsilon^{\frac{s}{1+s}}\frac{1}{|y_1|^{\frac{2s}{s+1}}}\bigg)dx_1dy_1.
\end{equation}
We now choose a constant $p_0$ satisfying $1<p_0<\frac{s+1}{2s}$ and apply Corollary \ref{Cor 3.1 2}, then we have  
\begin{equation}\label{C^1 regularity 22}
\begin{aligned}
 &\big |I_1(x,y,r_0)\big |\leq C C(p_0)r_0\bigg(\int_{-1}^{1}\big(1+\epsilon^{\frac{s}{1+s}}\frac{1}{|r_0y_1|^{\frac{2s}{s+1}}}\big)^{p_0}dy_1\bigg)^{\frac{1}{p_0}}\\&
 \leq C(\epsilon,p_0) (r_0+r_0^{\frac{1-s}{s+1}})
\end{aligned}
\end{equation}
In particular, \eqref{C^1 regularity 22} implies \begin{equation}\label{C^1 regularity 23}
  lim_{r_0\rightarrow 0}\|I_1(r,\cdot,\cdot)\|_{L^{\infty}}=0 .
\end{equation}

\noindent\emph{Estimate for $I_2$:}\\

We can use the bound $|\phi_{\epsilon}|>|\psi_0|$ and  get that \begin{equation*}
    I_2(x,y,r_0)\leq C \frac{|(x,y)|}{r_0^2}\int_{\mathbb{T}^2} \frac{1}{|x_1y_1|^s}dx_1dy_1\leq C \frac{|(x,y)|}{r_0^2}.
\end{equation*} Thus, for a fixed $r_0$, \begin{equation}\label{C^1 regularity 2}
    \lim_{(x,y)\rightarrow 0}I_2(r_0,x,y)=0.
\end{equation}
By \eqref{Cor3.1 1}, \eqref{velocity estimate 2}, \eqref{C^1 regularity 23} and \eqref{C^1 regularity 2}, we get
\begin{equation*}
    lim_{(x,y)\rightarrow 0}|\nabla(\phi_{\epsilon}-\psi_0)|(x,y)=0,
\end{equation*}
and it follows that 
\begin{equation*}
    \phi_{\epsilon}\in C^1({\mathbb{T}}^2).
\end{equation*}
 \end{proof}
\subsubsection{On the $C^1$ convergence of $\phi_{\epsilon}$ to $\psi_0$}
In the previous section, we proved that $\phi_{\epsilon}\in C^1$. In this section, we will prove the following statement: \begin{theorem}
\label{C1regularitytheorem2}
    $\lim_{\epsilon\rightarrow 0}\|\phi_{\epsilon}-\psi_0\|_{C^1(\mathbb{T}^2)}=0$.
\end{theorem}
By the symmetry of Biot-Savart law and the symmetry of $\phi_{\epsilon}$ and $\psi_0$, it suffices to prove the lemma below:
\begin{lemma}\label{lemma3}
 we have 
\begin{equation}
    I=\int_{\mathcal{T}} \frac{|P_{\epsilon}(\phi_{\epsilon})+1)|(x_1,y_1)}{|(x-x_1,y-y_1)|} dx_1 dy_1\leq C \epsilon^{\frac{s}{2}}.
\end{equation}
\end{lemma}
\begin{proof}
Denoting $\Omega_4=B_{C(s)\epsilon^{\frac{1}{2}}}\cap \mathcal{T}$ where $C(s)$ is given in Lemma \ref{maximum principle}, we have:
\begin{equation*}
    \begin{aligned}
&I=\int_{\Omega_4}\frac{|P_{\epsilon}(\phi_{\epsilon})+1|(x_1,y_1)}{|(x-x_1,y-y_1)|} dx_1y_1\\&+\int_{\mathcal{T}- \Omega_4}\frac{|P_{\epsilon}(\phi_{\epsilon})+1)|(x_1,y_1)}{|(x-x_1,y-y_1)|} dx_1y_1\\&:=I_1+I_2.    
    \end{aligned}
\end{equation*}

\noindent\emph{Estimate for $I_1$:}\\
In $\Omega_4$, if $\phi_{\epsilon}>\epsilon$, we have \begin{equation}\label{C^1 regularity 3}
    P_{\epsilon}(\phi_{\epsilon})=-1,
\end{equation} 
if $\phi_{\epsilon}<\epsilon$, then by Lemma \ref{maximum principle}, we have \begin{equation}\label{C^1 regularity 4}
    P_{\epsilon}(x,y)\leq C \epsilon^{\frac{s}{s+1}}\frac{1}{|(x,y)|^{\frac{2s}{s+1}}\sin^s(2\theta)}\leq C\epsilon^{\frac{s}{1+s}}\frac{1}{y^{\frac{2s}{s+1}}}.
\end{equation} By \eqref{C^1 regularity 3} and \eqref{C^1 regularity 4}, we have in
 $\Omega_4$,\begin{equation}\label{C^1 regularity 5}
    |P_{\epsilon}(\phi_{\epsilon})+1|(x,y)\leq C( \epsilon^{\frac{s}{1+s}}\frac{1}{y^{\frac{2s}{s+1}}}+1).
\end{equation}  By the explicit form of $I_1$, Lemma \ref{Steiner} and \eqref{C^1 regularity 5} implies that\begin{equation} \label{C^1 regularity 11}
\begin{aligned}
   & I_1\leq  C\sqrt{|\Omega_4|}+ C\int_{\Omega_4} \frac{1}{|(x-x_1,y-y_1)|}\epsilon^{\frac{s}{1+s}}\frac{1}{y_1^{\frac{2s}{s+1}}}dx_1dy_1\\&\leq  C\sqrt{\epsilon}+ C\epsilon^{\frac{s}{1+s}} \int_{B_{C(s)\epsilon^{\frac{1}{2}}}} \frac{1}{|(x-x_1,y-y_1)|} \frac{1}{y_1^{\frac{2s}{s+1}}}dx_1dy_1.
    \end{aligned}
\end{equation} Let $p_0$ be a fixed number such that $1<p_0<\frac{1+s}{2s}$, we may apply Corollary \ref{Del} with $p=p_0$ and we get \begin{equation}
\begin{aligned}
&\epsilon^{\frac{s}{1+s}} \int_{B_{C(s)\epsilon^{\frac{1}{2}}}} \frac{1}{|(x-x_1,y-y_1)|} \frac{1}{y_1^{\frac{2s}{s+1}}}dx_1dy_1 \\& \leq C(s) \epsilon^{\frac{s}{1+s}} \epsilon^{\frac{1}{2}} (\int_{-1}^{1}(\frac{1}{C(s)\epsilon^{\frac{1}{2}}|y|})^{\frac{2sp}{s+1}}dy)^{\frac{1}{p}}\\&\leq A(s) \epsilon^{\frac{1}{2}}.
\end{aligned}
\end{equation}  
Then by \eqref{C^1 regularity 11}, we have \begin{equation}\label{C^1 regularity 12}
   I_1\leq C(s)\epsilon^{\frac{1}{2}}. 
\end{equation}
\noindent\emph{Estimate for $I_2$:}\\
Note in $\mathcal{T}-\Omega_4$, by Lemma \ref{level set} and Lemma \ref{maximum principle}, define $D=\left\{(x,y): C(s)\epsilon^{\frac{1}{2}}<x<\frac{1}{2}, 0<y<\frac{C\epsilon}{x}\right\}$, we have that
\begin{equation}\label{C^1 regularity 8}
\begin{aligned}
   &\int_{\mathcal{T}-\Omega_4}\frac{1}{|(x-x_1,y-y_1)|}|P_{\epsilon}(\phi_{\epsilon})+1|(x_1,y_1)dx_1dy_1\\&= \int_{D}\frac{1}{|(x-x_1,y-y_1)|}|P_{\epsilon}(\phi_{\epsilon})+1|(x_1,y_1)dx_1dy_1\\& \leq C\int_{D}\frac{1}{|(x-x_1,y-y_1)|} \frac{\epsilon^{s}}{x_1^sy_1^s}dx_1y_1
   \end{aligned}
\end{equation} 
Define $D_i=B_{C\epsilon^{\frac{1}{2}}}(i\epsilon^{\frac{1}{2}},0)$, fix a large integer $N=\lfloor\frac{10}{\epsilon^{\frac{1}{2}}}\rfloor$, and we have \begin{equation}\label{cover}
    D\subseteq \cup_{i}^{N}D_{i}.
\end{equation}  
   By \eqref{cover}, we have \begin{equation}
\begin{aligned}
   &\int_{\mathcal{T}-\Omega_4}\frac{1}{|(x-x_1,y-y_1)|}|P_{\epsilon}(\phi_{\epsilon})+1|(x_1,y_1)dx_1dy_1\\&\leq C\int_{D}\frac{1}{|(x-x_1,y-y_1)|} \frac{\epsilon^{s}}{x_1^sy_1^s}dx_1y_1\\&\leq \epsilon^{s}\sum_{i=1}^{N}  C\int_{D_{i}}\frac{1}{|(x-x_1,y-y_1)|}\frac{1}{(Ci\epsilon^{\frac{1}{2}})^{s}}\frac{1}{y_1^s} dx_1dy_1\\&= \epsilon^{\frac{s}{2}}\sum_{i=1}^{N}  C\int_{D_{i}}\frac{1}{|(x-x_1,y-y_1)|}\frac{1}{i^{s}}\frac{1}{y_1^s} dx_1dy_1.
   \end{aligned}
\end{equation} 
  Then we apply Corollary \ref{Del} to $D_i$ for $1\leq i\leq N$, and we get \begin{equation}
      \begin{aligned}
      &\int_{\mathcal{T}-\Omega_4}\frac{1}{|(x-x_1,y-y_1)|}|P_{\epsilon}(\phi_{\epsilon})+1|(x_1,y_1)dx_1dy_1\\&\leq C\epsilon^{\frac{s}{2}} \sum_{i=1}^{N} \frac{1}{i^s} \epsilon^{\frac{1}{2}}(\int_{-1}^{1}(\frac{1}{C\epsilon^{\frac{1}{2}}|y|})^{ps}dy)^{\frac{1}{p}}\\&= C\sum_{i=1}^{\frac{10}{\epsilon^{\frac{1}{2}}}}\frac{1}{i^s}\epsilon^{\frac{1}{2}}\\&\leq C(s)\epsilon^{\frac{s}{2}},
      \end{aligned}
  \end{equation}
  for some $1<p<\frac{1}{s}$. 
  Then we have 
  \begin{equation}\label{C^1 regularity 20}
      I_2 \leq C(s) \epsilon^{\frac{s}{2}}.
  \end{equation}
By \eqref{C^1 regularity 12} and \eqref{C^1 regularity 20}, we finish the proof of Lemma \ref{lemma3}.
\end{proof}
Lemma \ref{C1regularitytheorem2} follows from Lemma \ref{lemma3}, and we now finish  the existence and uniqueness part of the proof of Theorem \ref{Algebraic singularity} in the case $\frac{1}{2}\leq s<1$.

\subsubsection{Lagrangian trajectories}
In the previous sections, we have shown the $C^1$ convergence of $\phi_{\epsilon}$ to $\psi_0$ as $\epsilon \rightarrow 0$ when $0<s<1$.  In this section, we will finish the proof of Theorem \ref{Algebraic singularity} by proving there are multiple particle trajectories that cross the origin in our singular steady state. We give the statement below:
\begin{theorem}\label{small scale}
For $\epsilon$ sufficiently small
 we have the following estimate for the vertical velocity associated to $\phi_{\epsilon}$:
 \begin{equation}
    u_{\epsilon}^{2}(0,y_1)=\frac{\partial \phi_{\epsilon}}{\partial x}(0,y_1)\geq C_{2} y_1^{1-s}, \text{for small and positive $y_1$.}
 \end{equation}
 As a direct consequence, non-trivial trajectories can arrive at and leave the origin in finite time.
\end{theorem} 
\begin{proof}[Proof of Theorem \ref{small scale}]
We choose a small constant $\delta_1$ such that    
\begin{equation}\label{delta 1}
\begin{aligned}
    & 0<2\delta_1<\epsilon.\\
\end{aligned}
\end{equation}
Since $\delta_1$ is so small, for $(x,y)\in B_{\delta_1}\cap\mathcal{T}$, we have (since $\phi_\epsilon(0,0)=\nabla\phi_\epsilon(0,0)=0$):
\begin{equation}\label{velocitypsi1}
    \begin{aligned}
        \phi_{\epsilon}(x,y)\leq 2(x^2+y^2)^{\frac{1}{2}}
    \end{aligned}
\end{equation}
and \begin{equation}\label{velocitypsi2}
        \phi_{\epsilon}(x,y)>\psi_0(x,y)> 0,
\end{equation}
by the maximum principle.
By the symmetry of $\phi_{\epsilon}$ and \eqref{velocitypsi1}, we have on $B_{\delta_1}(0)\cap \mathbb{T}^{2+}$,\begin{equation}\label{delta 2}
   0<\phi_{\epsilon}<\epsilon.
\end{equation}  As a result, from \eqref{velocitypsi1}-\eqref{velocitypsi2}, we have on $B_{\delta_1}(0)\cap \mathbb{T}^{2+}$ that
\begin{equation}\label{Lower bound}
    P_{\epsilon}(\phi_{\epsilon}(x,y))\leq \frac{-\epsilon^{s}}{(x^2+y^2)^{\frac{s}{2}}}.
\end{equation}
From the symmetry assumptions on $P_{\epsilon}(\phi_{\epsilon})$ and $\psi_{\epsilon}$, using Lemma \ref{GreensFunction} and Lemma \ref{Uniform upper bound}, we have
\begin{equation}
    \begin{aligned}
        &\quad u_{\epsilon}^{2}(0,y_1)\\&=\int_{\mathbb{T}^2}\frac{\partial G}{\partial x_1}((0,y_1),(x,y))P_{\epsilon}(\phi_{\epsilon}(x,y))dxdy\\&=O(y_1)+\int_{\mathbb{T}^2}-\frac{1}{2\pi}\frac{x}{x^2+(y-y_1)^2}P_{\epsilon}(\phi_{\epsilon}(x,y))dxdy\\&=O(y_1)+\frac{1}{2\pi}\int_{\mathbb{T}^{2+}}\frac{-8xyy_1}{(x^2+y^2+y_1^2)^2-4y_1^2x^2} P_{\epsilon}(\phi_{\epsilon}(x,y)) dxdy.
    \end{aligned}
\end{equation}
As a result, \begin{equation}\label{velocity}
    \begin{aligned}
    &\quad u_{\epsilon}^{2}(0,y_1)\\&=O(y_1)+\frac{1}{2\pi}\int_{\mathbb{T}^{2+}\cap B_{\delta_1}(0)}\frac{-8xyy_1}{(x^2+y^2+y_1^2)^2-4y_1^2x^2}P_{\epsilon}(\phi_{\epsilon}(x,y))dxdy\\&+\frac{1}{2\pi}\int_{\mathbb{T}^{2+}\cap B_{\delta_1}^{c}(0)}\frac{-8xyy_1}{(x^2+y^2+y_1^2)^2-4y_1^2x^2}P_{\epsilon}(\phi_{\epsilon}(x,y))dxdy\\&=O(y_1)+L_1+L_2.
    \end{aligned}
\end{equation}
\noindent\emph{Estimate on $L_1$}:\\ Let $y_1<\frac{\delta_1}{10}$,
by \eqref{Lower bound}, we have 
\begin{equation}\label{L1}
\begin{aligned}
      L_1&>\frac{4\epsilon^{s}y_1}{\pi}\int_{\mathbb{T}^{2+}\cap B_{\delta_1}(0)}\frac{xy}{(x^2+y^2+y_1^2)^2} \frac{1}{(x^2+y^2)^{\frac{s}{2}}}dxdy\\&>\frac{4\epsilon^{s}y_1}{\pi}\int_{\{y_1^2<x^2+y^2<100y_1^2\}}\frac{xy}{(x^2+y^2+y_1^2)^2} \frac{1}{(x^2+y^2)^{\frac{s}{2}}}dxdy\\&\geq C_{2} y_1\int_{y_1}^{10y_1}\frac{1}{r^{1+s}}dr\geq  C_{2} y_1^{1-s}.
\end{aligned}
\end{equation}
\\\noindent\emph{Estimate on $L_2$:}\\
By Lemma \ref{Uniform upper bound},
\begin{equation}\label{L2}
    \begin{aligned}
 |L_2|&\leq \frac{2y_1}{\pi}\int_{\mathbb{T}^{2+}\cap B_{\delta_1}^{c}(0)}\frac{1}{[x^2+y^2]}|P_{\epsilon}(\phi_{\epsilon}(x,y))|dxdy\\&<\frac{2y_1}{\pi \delta_1^2}\int_{\mathbb{T}^{2+}} |P_{\epsilon}(\phi_{\epsilon})(x,y)|dxdy \leq  C_{2} y_1.        
    \end{aligned}
\end{equation}
 
Combine \eqref{velocity}, \eqref{L1} and \eqref{L2}, we finish the proof of Theorem \ref{small scale}.
\end{proof}

\appendix
\section{}
We now proceed to prove some  technical results that we used in the course of proving the main theorems. First, let us recall the following standard fact (see \cite{Denisov2012DoubleEG}). 
\begin{lemma}\label{GreensFunction}
The Green's function for the Laplacian on $\mathbb{T}^2$ can be written as  \begin{equation}\label{Green}
        G(x,y)=\frac{1}{2\pi}\ln{\left|x-y\right|}+f(x-y),
    \end{equation}  where $f$ is smooth, and $\left|.\right|$ is the distance in $\mathbb{T}^2$.
\end{lemma}

\subsection{Proof of the property of $\psi_0$}\ \\
We begin with establishing the Properties from Section \ref{Construction1}.  The symmetry follows from the explicit form of \begin{equation*}
    \triangle^{-1}[-\text{sgn}(x)\text{sgn}(y)].
\end{equation*} Since 
we have \begin{equation}
\begin{aligned}
    &\triangle \psi_0(x,y)=-1,\text{for $(x,y)\in \mathbb{T}^{2+}$}.\\
    &\psi_0(x,y)=0, \text{for $(x,y)\in \partial \mathbb{T}^{2+}$}.
\end{aligned}
    \end{equation}
 By maximum principle, $\psi_0$ is positive on $\mathbb{T}^{2+}$.\\
\eqref{asytop} follows from explicit calculation and was given in \cite{Denisov2012DoubleEG}.\\
\subsection{Proof of Lemma \ref{Barrier function}}
\begin{proof}[Proof of Lemma \ref{Barrier function}]
By the scaling and symmetry of \eqref{special}, it behooves us to take the ansatz \begin{equation}\label{barrier}
    \tilde{\psi}=\epsilon^{\frac{s}{s+1}}r^{\frac{2}{s+1}}K(\theta),\text{with $K(\theta)=K(\frac{\pi}{2}-\theta)$}.
\end{equation}  \eqref{special} then gives the non-linear ODE 
\begin{equation}\label{Raw}
\begin{aligned}
 &\frac{4}{(1+s)^2}K+K^{''}=-\frac{1}{K^s}.\\
 &K(0)=0, K^{'}(\frac{\pi}{4})=0.
\end{aligned}
\end{equation}
 Multiplying \eqref{Raw} by $K^{'}$, we find a first integral of \eqref{Raw}:\begin{equation*}
     \Big(\frac{K^{'2}}{2}+\frac{2}{(1+s)^2}K^{2}+\frac{2}{1-s}K^{1-s}\Big)^{'}=0.
 \end{equation*}
By setting  \begin{equation*}
    K^{'}(0)=\sqrt{\frac{4}{(1+s)^2}B^2+\frac{2}{1-s}B^{1-s}},
\end{equation*}
 it now suffices to find a solution $(K,B)$ to \begin{equation}\label{ODE}
\begin{aligned}
     &\frac{4}{(1+s)^2}K+K^{''}=-\frac{1}{K^s},\\
     &K(0)=0,
     K(\frac{\pi}{4})=B,\\
    &K^{'}(0)=\sqrt{\frac{4}{(1+s)^2}B^2+\frac{2}{1-s}B^{1-s}}.
     \end{aligned}
\end{equation} 
While the existence of solution to \eqref{ODE} is equivalent to \begin{equation}
    I(B)=\int_{0}^{B}\frac{dk}{\sqrt{\frac{4}{(1+s)^2}(B^2-k^2)+\frac{2}{1-s}(B^{1-s}-k^{1-s})}}=\frac{\pi}{4}.
\end{equation}
Letting $k=B\tilde{k}$, we have \begin{equation*}
    I(B)=\int_{0}^{1}\frac{d\tilde {k}}{\sqrt{\frac{4}{(1+s)^2}(1-\tilde{k}^2)+B^{-1-s}(\frac{2}{1-s}-\frac{2\tilde{k}^{1-s}}{1-s})}}.
\end{equation*}
We have
\begin{equation*}
    I(0)=0,  I(\infty)=\frac{(1+s)\pi}{4}>\frac{\pi}{4}.
\end{equation*}
Moreover,by the dominated convergence theorem, $I(B)\in C(\mathbb{R}^{+})$.  Thus by continuity, we can find $B_0>0$ such that $I(B_0)=\frac{\pi}{4}$. Consequently, we have constructed a solution to \eqref{ODE}.
Based on our construction, we have K is a $C^1$ function even symmetric to $\frac{\pi}{4}$ with \begin{equation}
    K^{'}(\theta)>0,\text{for $\theta\in [0,\frac{\pi}{4}),$} 
\end{equation} as a result \begin{equation*}
    \frac{1}{C}\sin(2\theta)<K(\theta) <C\sin(2\theta).
\end{equation*}  We then finish the proof of \eqref{maxiun principle 18} by \eqref{barrier}.
\end{proof}
\section*{Acknowledgement}
The authors thank Siming He for providing various suggestions in writing the paper.
Both authors were partially supported by the NSF grants DMS 2043024 and DMS 2124748. T.M.E.
was partially supported by an Alfred P. Sloan Fellowship.
\section*{Data availability}
    Data sharing not applicable to this article as no data sets were generated or analyzed
during the current study.
\bibliographystyle{plain}
\bibliography{main}
\begin{itemize}
    \item Department of mathematics, Duke University, Durham, NC\\
E-mail address: tarek.elgindi@duke.edu
    \item Department of mathematics, Duke University, Durham, NC\\
    Email address: yh298@duke.edu
\end{itemize}
\end{document}